\documentclass[reqno,11pt]{article} 

\usepackage[left=2.5cm,right=2.5cm,top=3cm,bottom=3cm,a4paper]{geometry}
\usepackage{amsmath,amssymb}
\usepackage{amsthm, amscd} 
\usepackage{empheq}
\usepackage[all,cmtip]{xy}
\usepackage{float}
\usepackage{caption}
\usepackage{color}
\usepackage{rotating}
\usepackage[utf8]{inputenc}
\usepackage[T1]{fontenc}
\usepackage{enumitem}
\usepackage{url}

\usepackage[titles]{tocloft}
\makeatletter
\DeclareRobustCommand{\rvdots}{%
  \vbox{
    \baselineskip4\p@\lineskiplimit\z@
    \kern-\p@
    \hbox{.}\hbox{.}\hbox{.}
  }}
\makeatother

\newcommand{\Mod}[1]{\ (\textup{mod}\ #1)}
\makeatletter
\def\moverlay{\mathpalette\mov@rlay}
\def\mov@rlay#1#2{\leavevmode\vtop{%
   \baselineskip\z@skip \lineskiplimit-\maxdimen
   \ialign{\hfil$\m@th#1##$\hfil\cr#2\crcr}}}
\newcommand{\charfusion}[3][\mathord]{
    #1{\ifx#1\mathop\vphantom{#2}\fi
        \mathpalette\mov@rlay{#2\cr#3}
      }
    \ifx#1\mathop\expandafter\displaylimits\fi}
\makeatother

\theoremstyle{plain} 
\newtheorem{theorem}{\indent\sc Theorem}[section]
\newtheorem{lemma}[theorem]{\indent\sc Lemma}

\newtheorem{proposition}[theorem]{\indent\sc Proposition}

\theoremstyle{definition} 
\newtheorem{definition}[theorem]{\indent\sc Definition}

\newtheorem{remark}[theorem]{\indent\sc Remark}

\makeatletter
\def\address#1#2{\begingroup
\noindent\parbox[t]{7.8cm}{%
\small{\scshape\ignorespaces#1}\par\vskip1ex
\noindent\small{\itshape E-mail address}%
\/: #2\par\vskip4ex}\hfill%
\endgroup}%
\makeatother

\title{Adelic framed form class groups and explicit class field theory}
\author{
\textsc{Ja Kyung Koo, Dong Hwa Shin and Dong Sung Yoon} 
}
\date{} 
\begin{document}

\allowdisplaybreaks

\maketitle

 \footnote{
\thanks{}
}

\begin{abstract}
Let $D$ be a negative discriminant, and let $K=\mathbb{Q}(\sqrt{D})$. 
Let $\mathcal{Q}(D)$ denote the set of primitive positive definite binary quadratic
forms over $\mathbb{Z}$ of discriminant $D$. 
We introduce the set of adelic framed forms
\begin{equation*}
\widehat{\mathcal{Q}}(D)=
\left\{(Q,\,\gamma)\in
\mathcal{Q}(D)\times\mathrm{SL}_2(\widehat{\mathbb{Z}})~|~
Q\left(\gamma\begin{bmatrix}1\\0\end{bmatrix}\right)\in
\widehat{\mathbb{Z}}^\times\right\}
\end{equation*}
and its orbit space $\widehat{C}(D)$ under the natural action of $\mathrm{SL}_2(\mathbb{Z})$.
We define an explicit adelic analogue of the Gauss-Dirichlet composition law
on $\widehat{C}(D)$ and endow $\widehat{C}(D)$ with the quotient topology induced 
by the subspace topology on $\widehat{\mathcal{Q}}(D)$
inherited from the product topology
on $\mathcal{Q}(D)\times\mathrm{SL}_2(\widehat{\mathbb{Z}})$,
where
$\mathcal{Q}(D)$ is discrete
and $\mathrm{SL}_2(\widehat{\mathbb{Z}})$ has its profinite topology. 
We then prove that 
there is an isomorphism of topological groups
\begin{equation*}
\widehat{C}(D)\simeq\mathrm{Gal}\left(K^\mathrm{ab}(\mathfrak{t}^{1/\infty})/K(\mathfrak{t})\right),
\end{equation*}
where the Galois group is endowed with the Krull topology, 
$\mathfrak{t}$ is a positive transcendental real number, and $\mathfrak{t}^{1/\infty}=\{\sqrt[N]{\mathfrak{t}}~|~N\geq1\}$. 
Moreover, we identify an explicitly defined subgroup of 
$\widehat{C}(D)$ with
$\mathrm{Gal}(K^\mathrm{ab}/K)$ and 
describe the corresponding Galois action on special values of modular functions.
In this way, classical Gauss composition, finite-level form class groups, and Shimura reciprocity are brought together within a single adelic framework.
Finally, we show that the abstract group structure of 
$\widehat{C}(D)$ uniquely determines the imaginary quadratic field $K$.
\end{abstract}

\tableofcontents

\section {Introduction}

Let $K$ be a number field, that is, a
finite extension of the field $\mathbb{Q}$ of
rational numbers. Classical class field theory, 
as developed by Takagi \cite{Takagi}, 
describes finite abelian extensions of $K$ in terms of 
generalized ideal class groups attached to moduli;
see also \cite{Janusz}.
Let $K^\mathrm{ab}$ denote the maximal abelian extension of $K$. 
To describe the absolute abelian Galois group $\mathrm{Gal}(K^\mathrm{ab}/K)$, however,
one must pass to the inverse limit of the
corresponding finite Galois groups as the modulus varies.
\par
Chevalley's idelic formulation encodes the entire
theory in a single adelic object.  If
\begin{equation*}
C_K=\mathbb{A}_K^{\times}/K^{\times}
\end{equation*}
is the idele class group of $K$, then 
the global Artin reciprocity map $\mathrm{rec}_K:C_K\rightarrow\mathrm{Gal}(K^\mathrm{ab}/K)$
induces a canonical isomorphism
$C_K/N_{L/K}(C_L)\stackrel{\sim}{\rightarrow}\mathrm{Gal}(L/K)$
for every finite abelian extension $L/K$. 
Moreover, $\mathrm{rec}_K$ is surjective and its kernel is the 
connected component $C_K^\circ$ containing the identity. Hence
\begin{equation*}
 C_K/C_K^\circ\simeq\mathrm{Gal}(K^\mathrm{ab}/K)
\end{equation*}
as topological groups; see \cite{Chevalley, Milne, Neukirch}.  
This structural completeness does not, however, 
by itself yield a concrete arithmetic realization of class field theory.
For imaginary quadratic fields, the theory of complex multiplication provides
such a realization. 
\par
We now specialize the case
where $K$ is an imaginary quadratic field. 
The main theorems of complex multiplication show
that special values of modular functions at CM points 
generate abelian extensions of $K$, while
Shimura's reciprocity law describes 
the Galois action on these values
in terms of modular transformations;
see \cite{Deuring, Hasse, Shimura, Stevenhagen}. 
On the other hand, 
the Gauss-Dirichlet composition \cite{Dirichlet, Gauss}
equips the proper equivalence classes of primitive positive definite binary quadratic forms
with a concrete group law.
The resulting form class group can be identified with the ideal class group of an
order in $K$; see \cite{Cox}. 
Thus the adelic and form-theoretic 
descriptions have complementary strengths: the former captures 
the full abelian Galois group in adelic terms, 
whereas the latter, together with the theory of
complex multiplication, furnishes concrete representatives, 
special values, and transformation laws. 
This motivates the problem of constructing a single form-theoretic object that combines the adelic scope of the idele class group with the
concrete arithmetic content of Gauss-Dirichlet composition and Shimura reciprocity. 
\par
Recently, finite-level versions of this problem have been investigated in 
\cite{E-K-S, J-K-S-Y}. The form class groups associated with $\Gamma_1(N)$ realize ray class groups of 
orders in $K$, whereas those associated with $\Gamma(N)$ realize certain Galois groups over $K(\mathfrak{t})$ involving both ray class fields and the Kummer extension generated by $\sqrt[N]{\mathfrak{t}}$, where $\mathfrak{t}$ is a positive transcendental real number. 
As $N$ varies, these groups form an inverse system,
and the corresponding inverse-limit description 
in terms of CM-points on modular curves was established in \cite{J-K-S}. 
Related adelic structures arising from towers of modular curves
were investigated in the work of Daw and Zilber \cite{D-Z}.
The inverse-limit description, however, 
does not itself provide an orbit-space
realization in terms of individual quadratic-form objects
over $\widehat{\mathbb{Z}}$ or $\mathbb{Z}\times\widehat{\mathbb{Z}}$,
where
\begin{equation*}
\widehat{\mathbb{Z}}=\varprojlim_{N\geq1}\mathbb{Z}/N\mathbb{Z}
\quad\left(\simeq\prod_{p\,:\,\text{primes}}
\mathbb{Z}_p\right)
\end{equation*}
denotes the profinite 
completion of $\mathbb{Z}$.
\par
In this paper, 
we construct a single adelic framed form class group. More precisely, for a negative discriminant $D$
we consider
\begin{equation*}
\widehat{\mathcal{Q}}(D)=
\left\{(Q,\,\gamma)\in\mathcal{Q}(D)\times
\mathrm{SL}_2(\widehat{\mathbb{Z}})~|~
Q\left(\gamma\begin{bmatrix}1\\0\end{bmatrix}\right)\in\widehat{\mathbb{Z}}^{\times}\right\},
\end{equation*}
where $\mathcal{Q}(D)$ denotes the set of primitive positive definite binary quadratic forms over
$\mathbb{Z}$ of discriminant $D$. We then take the quotient 
of this set by the natural action of $\operatorname{SL}_2(\mathbb{Z})$. The integral form $Q$ retains the classical arithmetic data, while the frame $\gamma$ records compatible changes of variables at all finite levels.
On the resulting quotient space $\widehat{C}(D)$, we define
a direct adelic analogue of Gauss-Dirichlet composition 
(Definition \ref{explicit}). The natural projections to finite levels,
together with finite-level reciprocity isomorphisms, induce an isomorphism
\begin{equation*}
\eta:\widehat{C}(D)
\stackrel{\sim}{\rightarrow}\mathrm{Gal}\left(K^\mathrm{ab}(\mathfrak{t}^{1/\infty})/K(\mathfrak{t})\right)
\end{equation*}
(Theorem \ref{main}). Here, $K=\mathbb{Q}(\sqrt{D})$. 
This construction also yields a form-theoretic formulation of
the Shimura reciprocity law (Theorem \ref{Shimurarec}). 
Moreover,  
when $\widehat{C}(D)$ and $\mathrm{Gal}
\left(K^\mathrm{ab}(\mathfrak{t}^{1/\infty})/K(\mathfrak{t})\right)$
are endowed with the quotient topology and the Krull topology, respectively,
$\eta$ is an isomorphism of topological groups (Theorem \ref{topological}). 
\par
The Kummer direction in the above nonabelian Galois group retains
arithmetic information that is not detected by $\mathrm{Gal}(K^\mathrm{ab}/K)$ alone.
Indeed, Onabe \cite{Onabe} and later Angelakis and Stevenhagen \cite{A-S}
showed that
nonisomorphic imaginary quadratic fields, even with different class numbers, 
may have isomorphic absolute abelian Galois groups.
In contrast, we prove that the underlying abstract group 
$\widehat{C}(D)$ uniquely determines the imaginary quadratic field. For negative
discriminants $D_1$ and $D_2$, one has
\begin{equation*}
\widehat{C}(D_1)\simeq\widehat{C}(D_2)~\textrm{as groups}
\quad\Longleftrightarrow\quad
\mathbb{Q}(\sqrt{D_1})=\mathbb{Q}(\sqrt{D_2})
\end{equation*}
(Theorem \ref{rigidity}).
\par
Thus our construction not only realizes the inverse system of
finite-level form class groups as a single adelic orbit space equipped
with an explicit composition law, but also produces a nonabelian
profinite group that determines the
imaginary quadratic field. This formulation may provide a useful
starting point for studying related infinite-level structures arising
from towers of modular curves.
\par
\bigskip
Throughout this paper, we use the following notation:
\begin{itemize}[noitemsep]
\item $N\in\mathbb{N}$, where $\mathbb{N}$ denotes the set of positive integers.
\item $D$ is a negative integer such that $D\equiv1$ or $0\Mod{4}$. 
\item $K=\mathbb{Q}(\sqrt{D})$ is an imaginary quadratic field.
\item $\mathcal{O}$ is the order in $K$ of discriminant $D$.
\item $\mathfrak{t}$ is a positive transcendental real number.
\end{itemize}

\section {Ray class fields for orders}

In this preliminary section, we briefly review 
the definition of ray class fields associated with orders in imaginary quadratic fields. 
\par
Let $I(\mathcal{O})$ be the group of proper fractional $\mathcal{O}$-ideals 
(see \cite[$\S$7.A]{Cox}). 
A nonzero integral $\mathcal{O}$-ideal $\mathfrak{a}$ is said to be prime to 
$N$ if $\mathfrak{a}+N\mathcal{O}=\mathcal{O}$. 
Let $I(\mathcal{O},\,N)$ and $P_1(\mathcal{O},\,N)$ be the subgroups of $I(\mathcal{O})$ 
defined by
\begin{align*}
I(\mathcal{O},\,N)&=\bigl\langle\text{nonzero proper integral $\mathcal{O}$-ideals
prime to $N$}\bigr\rangle,\\
P_1(\mathcal{O},\,N)&=\bigl\langle\nu\mathcal{O}~|~
\nu\in\mathcal{O}\setminus\{0\},~
\nu\equiv1\Mod{N\mathcal{O}}\bigr\rangle,
\end{align*}
respectively. The quotient group 
\begin{equation*}
C(\mathcal{O},\,N)=
I(\mathcal{O},\,N)/P_1(\mathcal{O},\,N)
\end{equation*}
is isomorphic to a generalized 
ideal class group modulo $\ell_\mathcal{O} N\mathcal{O}_K$,
where $\mathcal{O}_K$ is the maximal order of  $K$ and
$\ell_\mathcal{O}=[\mathcal{O}_K:\mathcal{O}]$ is the conductor of $\mathcal{O}$; 
see \cite[Theorem 3.18]{Schertz}. By the existence
theorem of class field theory (see
\cite[$\S$V.9]{Janusz} and \cite[$\S$2]{J-K-S-Y}), there exists
a unique finite abelian extension $K_{\mathcal{O},\,N}$ of $K$
such that every nonzero prime ideal of $\mathcal{O}_K$ that ramifies in 
$K_{\mathcal{O},\,N}$ divides $\ell_\mathcal{O} N\mathcal{O}_K$,
and the Artin map for the modulus $\ell_\mathcal{O} N\mathcal{O}_K$ induces an isomorphism 
\begin{equation*}
C(\mathcal{O},\,N)\stackrel{\sim}{\rightarrow}
\mathrm{Gal}(K_{\mathcal{O},\,N}/K). 
\end{equation*}
The field $K_{\mathcal{O},\,N}$ is called the 
\textit{ray class field of $\mathcal{O}$ modulo $N\mathcal{O}$}.
This notion was originally introduced by S\"{o}hngen \cite{Sohngen}
and was later discussed by Stevenhagen \cite{Stevenhagen} from
the viewpoint of 
explicit class field theory. Let 
$K^\mathrm{ab}$ denote the maximal abelian extension of $K$.

\begin{lemma}\label{basic}
Concerning ray class fields of $\mathcal{O}$ modulo $N\mathcal{O}$, we have the following properties.
\begin{enumerate}
\item[\textup{(i)}] If $M,\,N\in\mathbb{N}$ and $N\,|\,M$, then
$K_{\mathcal{O},\,N}\subseteq K_{\mathcal{O},\,M}$.
\item[\textup{(ii)}] We get
$\displaystyle
\bigcup_{N\in\mathbb{N}}K_{\mathcal{O},\,N}=K^\mathrm{ab}$.
\end{enumerate}
\end{lemma}
\begin{proof}
\begin{enumerate}
\item[(i)] See \cite[$\S$4, p. 169]{Stevenhagen}.
\item[(ii)] See \cite[Lemma 4.1 (3)]{J-K-S}.  
\end{enumerate}
\end{proof}

Let $\mathbb{H}=\{\tau\in\mathbb{C}~|~\mathrm{Im}(\tau)>0\}$ be the
complex upper half-plane, and let
\begin{equation*}
\Gamma(N)=\{\alpha\in\mathrm{SL}_2(\mathbb{Z})~|~\alpha\equiv I_2\Mod{NM_2(\mathbb{Z})}\}
\end{equation*}
be the principal congruence subgroup of level $N$. 
Let $\mathcal{F}_N$ be the field of meromorphic modular functions for 
$\Gamma(N)$ defined on $\mathbb{H}$
whose Fourier expansions with respect to 
$q^{1/N}$ have coefficients in the $N$th cyclotomic field $\mathbb{Q}(\zeta_N)$,
where
\begin{equation*}
q=e^{2\pi\mathrm{i}\tau}\quad\textrm{and}\quad\zeta_N=e^{2\pi\mathrm{i}/N}. 
\end{equation*}
Then $\mathcal{F}_N/\mathcal{F}_1$ is
a Galois extension and
\begin{equation}\label{GFF}
\mathrm{Gal}(\mathcal{F}_N/\mathcal{F}_1)
\simeq\mathrm{GL}_2(\mathbb{Z}/N\mathbb{Z})/\langle-I_2\rangle
\end{equation}
(see \cite[Theorem 6.6 and Propositions 6.1 and 6.9 (1)]{Shimura}). 
We use the right action notation of $\mathrm{GL}_2(\mathbb{Z}/N\mathbb{Z})/\langle-I_2
\rangle$ on $\mathcal{F}_N$ as
\begin{equation*}
f\mapsto f^\gamma\quad(\gamma\in\mathrm{GL}_2(\mathbb{Z}/N\mathbb{Z})/\langle-I_2\rangle,~
f\in\mathcal{F}_N). 
\end{equation*}
More precisely,

\begin{proposition}\label{GFaction}
Let $\gamma\in\mathrm{GL}_2(\mathbb{Z}/N\mathbb{Z})/\langle-I_2\rangle$
and $f\in\mathcal{F}_N$. Let
\begin{equation*}
f(\tau)=\sum_{n\gg-\infty}c_nq^{n/N}~\textrm{with}~c_n\in\mathbb{Q}(\zeta_N)
\quad(\tau\in\mathbb{H}) 
\end{equation*}
be the Fourier expansion of $f$. 
\begin{enumerate}
\item[\textup{(i)}] If $\gamma$ is represented by $\begin{bmatrix}1&0\\0&d\end{bmatrix}$ 
with $d\in(\mathbb{Z}/N\mathbb{Z})^\times$, then 
\begin{equation*}
f^\gamma(\tau)=\sum_{n\gg-\infty}c_n^{\sigma_d}q^{n/N}\quad(\tau\in\mathbb{H}),
\end{equation*}
where $\sigma_d$ is the automorphism of $\mathbb{Q}(\zeta_N)$ defined by $\zeta_N^{\sigma_d}=\zeta_N^d$. 
\item[\textup{(ii)}] If $\gamma$ is represented by a matrix $\alpha\in\mathrm{SL}_2(\mathbb{Z})$, then
\begin{equation*}
f^\gamma=f\circ\alpha. 
\end{equation*}
\end{enumerate}
\end{proposition}
\begin{proof}
See \cite[Theorem 6.6 (3) and Proposition 6.21 (3)]{Shimura}. 
\end{proof}

\begin{remark}
When $\gamma\in\mathrm{GL}_2(\mathbb{Z}/N\mathbb{Z})$
with the image $[\gamma]$ in $\mathrm{GL}_2(\mathbb{Z}/N\mathbb{Z})/
\langle-I_2\rangle$, we usually write
$f^\gamma$ for $f^{[\gamma]}$ for simplicity. 
\end{remark}

Let $\tau_\mathcal{O}$ denote the element of $\mathbb{H}$ given by
\begin{equation*}
\tau_\mathcal{O}=
\left\{\begin{array}{cl}
\displaystyle\frac{-1+\sqrt{D}}{2}
& \text{if}~D\equiv1\Mod{4},\vspace{0.1cm}\\
\displaystyle\frac{\sqrt{D}}{2}& \text{if}~D\equiv0\Mod{4}, 
\end{array}\right.
\end{equation*}
so that $\mathcal{O}=\mathbb{Z}\tau_\mathcal{O}+\mathbb{Z}$
(see \cite[Lemma 7.2]{Cox}). 
Furthermore, we let
\begin{equation*}
\mathcal{F}_{N,\,\tau_\mathcal{O}}=\left\{f\in\mathcal{F}_N~|~
\text{$f$ is finite at $\tau_\mathcal{O}$}\right\}.
\end{equation*}
The theory 
of Shimura's canonical models provides the following 
description of the field $K_{\mathcal{O},\,N}$
in terms of special values of modular functions.

\begin{proposition}\label{generation}
We have 
$K_{\mathcal{O},\,N}=
K\left(f(\tau_\mathcal{O})~|~f\in\mathcal{F}_{N,\,\tau_\mathcal{O}}\right)$.
\end{proposition}
\begin{proof}
See \cite[Theorem 4]{Cho} and also \cite[p. 169]{Stevenhagen}.
\end{proof}

\begin{remark}
Note that $K_{\mathcal{O},\,N}$ contains $\zeta_N$. 
\end{remark}

\section {Form class groups isomorphic to ray class groups of orders}

In this section, we recall finite-level form class groups
and describe their group structures in terms of Galois groups.
\par
Let $\mathcal{Q}(D,\,N)$ be the
set of primitive positive definite binary quadratic forms over $\mathbb{Z}$ 
of discriminant $D$ whose coefficients of $x^2$ are relatively prime to $N$, that is,
\begin{equation*}
\mathcal{Q}(D,\,N)=\left\{ax^2+bxy+cy^2\in
\mathbb{Z}[x,\,y]~|~\gcd(a,\,b,\,c)=1,~b^2-4ac=D,~a>0,~\gcd(a,\,N)=1\right\}. 
\end{equation*}
We simply write $\mathcal{Q}(D)$ for 
$\mathcal{Q}(D,\,1)$. 
Specifically, we let
$Q_0=a_0x^2+b_0xy+c_0y^2$ be
the principal form in $\mathcal{Q}(D,\,N)$, namely,
\begin{equation*}
Q_0=a_0x^2+b_0xy+c_0y^2=\left\{\begin{array}{ll}
\displaystyle x^2+xy+\frac{1-D}{4}y^2& \text{if}~D\equiv1\Mod{4},\\
\displaystyle x^2-\frac{D}{4}y^2 & \text{if}~D\equiv0\Mod{4}. 
\end{array}\right.
\end{equation*}
The congruence subgroup
\begin{equation*}
\Gamma_1(N)=\left\{\gamma\in\mathrm{SL}_2(\mathbb{Z})~|~
\gamma\equiv\begin{bmatrix}1& \mathrm{*}\\0&1\end{bmatrix}\Mod{NM_2(\mathbb{Z})}\right\}
\end{equation*}
acts on the right
on the set $\mathcal{Q}(D,\,N)$ as
\begin{equation*}
Q^\alpha=Q\left(\alpha\begin{bmatrix}x\\y\end{bmatrix}\right)
\quad(Q\in\mathcal{Q}(D,\,N),~\alpha\in\Gamma_1(N)),
\end{equation*} 
which naturally yields an equivalence relation $\sim_{\Gamma_1(N)}$ on $\mathcal{Q}(D,\,N)$.
For $Q\in\mathcal{Q}(D,\,N)$ we write
its equivalence class as $[Q]_{\Gamma_1(N)}$.  
Denote the set of equivalence classes by $C_{\Gamma_1(N)}(D,\,N)$, namely,
\begin{equation*}
C_{\Gamma_1(N)}(D,\,N)=\mathcal{Q}(D,\,N)/\sim_{\Gamma_1(N)}
=\left\{[Q]_{\Gamma_1(N)}~|~Q\in\mathcal{Q}(D,\,N)\right\}.
\end{equation*}

\begin{definition}
For $Q=ax^2+bxy+cy^2\in\mathcal{Q}(D)$, we define
$\omega_Q$ to be the zero of the quadratic polynomial $Q(x,\,1)$
lying in $\mathbb{H}$. Moreover, if $Q\in\mathcal{Q}(D,\,N)$, then we let $m_{Q}$
be an element of $\mathrm{GL}_2(\mathbb{Z}/N\mathbb{Z})$ defined by 
\begin{equation*}
m_Q=m_{N,\,Q}=
\begin{bmatrix}1& -a^{-1}\left(\frac{b+b_0}{2}\right)\\
0&a^{-1}\end{bmatrix}
\end{equation*}
where $a^{-1}$ means the inverse of $a$ in $(\mathbb{Z}/N\mathbb{Z})^\times$.
\end{definition}

\begin{remark}\label{QQzz}
Let $Q=ax^2+bxy+cy^2,\,Q'\in\mathcal{Q}(D)$. 
\begin{enumerate}
\item[(i)] Since $b^2-4ac=D=b_0^2-4a_0c_0$,  $\displaystyle\frac{b+b_0}{2}$ 
and $\displaystyle\frac{b-b_0}{2}$ are integers.
\item[(ii)] 
Note that 
\begin{equation*}
\omega_Q=\frac{-b+\sqrt{D}}{2a}\quad\textrm{and}\quad
Q=a(x-\omega_Qy)(x-\overline{\omega}_Qy).
\end{equation*}
One sees that $Q=Q'$ if and only if $\omega_Q=\omega_{Q'}$. 
\item[(iii)] Moreover, if $Q\in\mathcal{Q}(D,\,N)$, then
the lattice $\mathbb{Z}\omega_Q+\mathbb{Z}$
belongs to $I(\mathcal{O},\,N)$ (see \cite[Theorem 7.7 (i) and (7.16)]{Cox}). 
\end{enumerate}
\end{remark}

For $Q\in\mathcal{Q}(D,\,N)$ and $\alpha\in\mathrm{SL}_2(\mathbb{Z})$, 
the notation $Q^\alpha$ means the action of $\alpha$ on $Q\in\mathcal{Q}(D)$. 

\begin{definition}\label{Noperation}
Define a binary operation on $C_{\Gamma_1(N)}(D,\,N)$
as follows: let 
\begin{align*}
\mathcal{C}&=[Q]_{\Gamma_1(N)}~\textrm{with}~
Q=ax^2+bxy+cy^2\in\mathcal{Q}(D,\,N),\\
\mathcal{C}'&=[Q']_{\Gamma_1(N)}~\textrm{with}~
Q'=a'x^2+b'xy+c'y^2\in\mathcal{Q}(D,\,N).
\end{align*}
By \cite[Lemmas 2.3 and 2.25]{Cox}, there exists a matrix $\alpha
=\begin{bmatrix}a_1&a_2\\a_3&a_4\end{bmatrix}\in\mathrm{SL}_2(\mathbb{Z})$ 
so that the form
\begin{equation*}
Q''=Q'^\alpha=a''x^2+b''xy+c''y^2
\end{equation*}
satisfies
\begin{equation*}
\gcd\left(a,\,a'',\frac{b+b''}{2}\right)=1. 
\end{equation*}
And, it follows from \cite[Lemma 3.2]{Cox} that 
there is an integer $B$ satisfying
\begin{align*}
B&\equiv b\Mod{2a},\\
B&\equiv b''\Mod{2a''},\nonumber\\
B^2&\equiv D\Mod{4aa''}.\nonumber
\end{align*}
Construct a Dirichlet composition $Q'''$ of $Q$ and $Q''$ as
\begin{equation*}
Q'''=aa''x^2+Bxy+\frac{B^2-D}{4aa''}y^2. 
\end{equation*}
Set
\begin{equation*}
\nu_1=\frac{\omega_{Q'''}}{a_3\omega_{Q'''}+a_4}
\quad\textrm{and}\quad
\nu_2=\frac{1}{a_3\omega_{Q'''}+a_4}
\end{equation*}
One can verify the existence of $(u,\,v)\in\mathbb{Z}^2$ and 
$\sigma\in\mathrm{SL}_2(\mathbb{Z})$ such that
\begin{equation*}
u\nu_1+v\nu_2=1\quad\textrm{and}\quad
\sigma\equiv\begin{bmatrix}\mathrm{*} & \mathrm{*}\\
u & v\end{bmatrix}\Mod{NM_2(\mathbb{Z})}. 
\end{equation*}
We then define
\begin{equation}\label{CCF}
\mathcal{C}\mathcal{C}'=\left[(Q''')^{\sigma^{-1}}\right]_{\Gamma_1(N)}. 
\end{equation}
\end{definition}

The above Definition \ref{Noperation} is given in \cite[Definition 5.7]{J-K-S-Y}, 
whose well-definedness is due to the next proposition. 
Recall from Proposition \ref{generation} that
$K_{\mathcal{O},\,N}=K(f(\tau_\mathcal{O})~|~f\in\mathcal{F}_{N,\,\tau_\mathcal{O}})$.

\begin{proposition}\label{isomorphism}
The binary operation on $C_{\Gamma_1(N)}(D,\,N)$
given in Definition \ref{Noperation}
is well defined and makes $C_{\Gamma_1(N)}(D,\,N)$
a group isomorphic to $C(\mathcal{O},\,N)$ via the mapping\begin{align*}
\phi_N:C_{\Gamma_1(N)}(D,\,N)&\rightarrow C(\mathcal{O},\,N)\\
[Q]_{\Gamma_1(N)}&\mapsto
[\mathbb{Z}\omega_Q+\mathbb{Z}]. 
\end{align*}
Furthermore, the mapping 
\begin{align*}
\sigma_N:C_{\Gamma_1(N)}(D,\,N)&\rightarrow\mathrm{Gal}(K_{\mathcal{O},\,N}/K)\\
[Q]_{\Gamma_1(N)}&\mapsto
\bigg(f(\tau_\mathcal{O})\mapsto f^{m_Q}(-\overline{\omega}_Q)~|~
f\in\mathcal{F}_{N,\,\tau_\mathcal{O}}\bigg)
\end{align*}
is a well-defined isomorphism,
where $\overline{\,\cdot\,}$ means
the complex conjugation.  
\end{proposition}
\begin{proof}
See \cite[Theorems 9.4 and 12.3]{J-K-S-Y}.
\end{proof}

\begin{remark}
Roughly speaking 
$\sigma_N$ is obtained by composing 
the Artin map modulo $\ell_\mathcal{O}N\mathcal{O}_K$
and $\phi_N$. 
See \cite[$\S$11--12]{J-K-S-Y} for details.  
\end{remark}

\begin{remark}
Let $j(\mathcal{O})$ be the $j$-invariant of an elliptic curve
with complex multiplication by $\mathcal{O}$. 
\begin{enumerate}
\item[(i)] Let $H_{\mathcal{O}}=K(j(\mathcal{O}))$, 
which is the ring class field of order $\mathcal{O}$.
Let $W_{\mathcal{O},\,N}$ be the Cartan subgroup of
$\mathrm{GL}_2(\mathbb{Z}/N\mathbb{Z})$ obtained from
$(\mathcal{O}/N\mathcal{O})^\times$ through its multiplication action
with respect to the ordered basis
$\{\tau_{\mathcal{O}}+N\mathcal{O},\,1+N\mathcal{O}\}$, and let
$U_{\mathcal{O},\,N}$ be the subgroup corresponding to the image of
$\mathcal{O}^\times$. 
Following Stevenhagen's practical formulation of the Shimura
reciprocity law \cite{Stevenhagen}, the map
\begin{equation*}
W_{\mathcal{O},\,N}/U_{\mathcal{O},\,N}
\rightarrow\mathrm{Gal}(K_{\mathcal{O},\,N}/H_{\mathcal{O}}),
\quad[\gamma]\mapsto
\bigg(f(\tau_{\mathcal{O}})\mapsto f^\gamma(\tau_{\mathcal{O}})
~|~f\in\mathcal{F}_{N,\,\tau_\mathcal{O}}\bigg)~(\gamma\in W_{\mathcal{O},\,N})
\end{equation*}
is a well-defined isomorphism.
\item[(ii)] 
Let $E/\mathbb{Q}(j(\mathcal{O}))$ be an elliptic curve with
complex multiplication by $\mathcal{O}$ and
$j(E)=j(\mathcal{O})$.
The natural action of
$G_{\mathbb{Q}(j(\mathcal{O}))}=\mathrm{Gal}\left(
\overline{\mathbb{Q}(j(\mathcal{O}))}/\mathbb{Q}(j(\mathcal{O}))\right)$
on the $N$-torsion subgroup $E[N]$ of $E$ induces a Galois representation
\begin{equation*}
\rho_{E,\,N}:G_{\mathbb{Q}(j(\mathcal{O}))}\rightarrow
\mathrm{Aut}(E[N])\simeq\mathrm{GL}_2(\mathbb{Z}/N\mathbb{Z}).
\end{equation*}
Here, $\overline{\mathbb{Q}(j(\mathcal{O}))}$ means
the algebraic closure of $\mathbb{Q}(j(\mathcal{O}))$ in $\mathbb{C}$. 
Using the algebraic theory of complex multiplication developed by Deuring \cite{Deuring} and others, 
Lozano-Robledo \cite{Lozano-Robledo} described the image of $\rho_{E,\,N}$ 
up to conjugacy in terms of the Cartan subgroup $W_{\mathcal{O},\,N}$ and its normalizer. 
\end{enumerate}
\end{remark}

\section {Form class groups induced by principal congruence subgroups}

We recall the finite-level form class group
associated with the principal congruence subgroup $\Gamma(N)$  and
its realization as a Galois group over $K(\mathfrak{t})$.
We also give an explicit binary operation on the form class group
compatible with this Galois-theoretic group structure.
\par
Since $K_{\mathcal{O},\,N}$ contains $\zeta_N$, 
we see that
$K_{\mathcal{O},\,N}(\sqrt[N]{\mathfrak{t}})/K_{\mathcal{O},\,N}(\mathfrak{t})$
is a Kummer extension, and
\begin{equation*}
\mathrm{Gal}\left(K_{\mathcal{O},\,N}(\sqrt[N]{\mathfrak{t}})/K_{\mathcal{O},\,N}(\mathfrak{t})\right)
\simeq\mathbb{Z}/N\mathbb{Z}
\end{equation*}
because $\mathfrak{t}$ is transcendental. 
Furthermore, 
\begin{equation*}
\mathrm{Gal}\left(K_{\mathcal{O},\,N}(\sqrt[N]{\mathfrak{t}})/K(\mathfrak{t})\right)
=\mathrm{Gal}\left(K_{\mathcal{O},\,N}(\sqrt[N]{\mathfrak{t}})/K_{\mathcal{O},\,N}(\mathfrak{t})\right)
\rtimes\mathrm{Gal}\left(K_{\mathcal{O},\,N}(\sqrt[N]{\mathfrak{t}})/K(
\sqrt[N]{\mathfrak{t}})\right)
\end{equation*}
(see \cite[Lemma 12.4]{J-K-S-Y}). 
\par
The principal congruence subgroup $\Gamma(N)$ acts on $\mathcal{Q}(D,\,N)$
in a similar way to $\Gamma_1(N)$, which induces
the equivalence relation $\sim_{\Gamma(N)}$ on $\mathcal{Q}(D,\,N)$. 
For $Q\in\mathcal{Q}(D,\,N)$, let $[Q]_{\Gamma(N)}$ be its equivalence class. 
Denote the set of equivalence classes by 
\begin{equation*}
C_{\Gamma(N)}(D,\,N)=\mathcal{Q}(D,\,N)/\sim_{\Gamma(N)}.
\end{equation*}

\begin{proposition}\label{isomorphism2}
The map
\begin{align*}
\psi_N:C_{\Gamma(N)}(D,\,N)&\rightarrow
\mathrm{Gal}\left(K_{\mathcal{O},\,N}(\sqrt[N]{\mathfrak{t}})/K(\mathfrak{t})\right)\\
[Q]_{\Gamma(N)}=
[ax^2+bxy+cy^2]_{\Gamma(N)}&\mapsto
\left(\begin{array}{lcl}
\sqrt[N]{\mathfrak{t}}&\mapsto& 
\zeta_N^{\frac{b-b_0}{2}}\sqrt[N]{\mathfrak{t}},\\
f(\tau_\mathcal{O})&\mapsto&
f^{m_Q}(-\overline{\omega}_Q)\quad(f\in\mathcal{F}_{N,\,\tau_\mathcal{O}})
\end{array}\right)
\end{align*}
is a well-defined bijection. Thus
$C_{\Gamma(N)}(D,\,N)$ can be endowed
with a group structure so that $\psi_N$ becomes an isomorphism. 
\end{proposition}
\begin{proof}
See \cite[Corollary 12.6]{J-K-S-Y}. 
\end{proof}

\begin{remark}\label{operationremark}
\begin{enumerate}
\item[(i)] We see from Propositions \ref{isomorphism} and \ref{isomorphism2} that
\begin{equation*}
\psi_N([Q]_{\Gamma(N)})|_{K_{\mathcal{O},\,N}}=\sigma_N
([Q]_{\Gamma_1(N)})\quad(Q\in\mathcal{Q}(D,\,N)). 
\end{equation*}
\item[(ii)] Let 
\begin{align*}
\mathcal{C}&=[Q]_{\Gamma(N)}~\textrm{with}~
Q=ax^2+bxy+cy^2\in\mathcal{Q}(D,\,N),\\
\mathcal{C}'&=[Q']_{\Gamma(N)}~\textrm{with}~
Q'=a'x^2+b'xy+c'y^2\in\mathcal{Q}(D,\,N).
\end{align*}
We find by Proposition \ref{isomorphism2} that
\begin{align*}
\left(\sqrt[N]{\mathfrak{t}}\right)^{\psi_N(\mathcal{C}\mathcal{C}')}&=
\left(\left(\sqrt[N]{\mathfrak{t}}\right)^{\psi_N(\mathcal{C})}\right)
^{\psi_N(\mathcal{C}')}\quad\textrm{because $\psi_N$ is a homomorphism}\\
&=\left(\zeta_N^{\frac{b-b_0}{2}}\sqrt[N]{\mathfrak{t}}\right)
^{\psi_N(\mathcal{C}')}\\
&=\left(\zeta_N^{\frac{b-b_0}{2}}\right)^{\psi_N(\mathcal{C}')}\left(\sqrt[N]{\mathfrak{t}}\right)
^{\psi_N(\mathcal{C}')}\\
&=\left(\zeta_N^{a'^{-1}\left(\frac{b-b_0}{2}\right)}\right)
\left(\zeta_N^{\frac{b'-b_0}{2}}\sqrt[N]{\mathfrak{t}}\right)
\quad\textrm{by Proposition \ref{GFaction}, where}\\
&\hspace{5.5cm}\textrm{$a'^{-1}$ is the inverse of $a'$ in $(\mathbb{Z}/N\mathbb{Z})^\times$}\\
&=\zeta_N^{a'^{-1}\left(\frac{b-b_0}{2}\right)+\left(\frac{b'-b_0}{2}\right)}\sqrt[N]{\mathfrak{t}}. 
\end{align*}
\end{enumerate}
\end{remark}

Using Definition \ref{Noperation} and Remark \ref{operationremark}, one can present
a well-defined binary operation on $C_{\Gamma(N)}(D,\,N)$ in an explicit way. 

\begin{definition}\label{Noperation2}
Let
\begin{align*}
\mathcal{C}&=[Q]_{\Gamma(N)}~\textrm{with}~
Q=ax^2+bxy+cy^2\in\mathcal{Q}(D,\,N),\\
\mathcal{C}'&=[Q']_{\Gamma(N)}~\textrm{with}~
Q'=a'x^2+b'xy+c'y^2\in\mathcal{Q}(D,\,N).
\end{align*}
Let 
\begin{equation*}
F=\widetilde{a}x^2+\widetilde{b}xy+\widetilde{c}y^2
\end{equation*}
be a form in $\mathcal{Q}(D,\,N)$ obtained by the construction in Definition \ref{Noperation} 
(that is, $F=(Q''')^{\sigma^{-1}}$ given in \eqref{CCF}) so that
\begin{equation*}
[Q]_{\Gamma_1(N)}
\cdot[Q']_{\Gamma_1(N)}=[F]_{\Gamma_1(N)}.
\end{equation*}
Take an integer $\ell$ such that
\begin{equation*}
\widetilde{a}\ell
\equiv
a'^{-1}\left(\frac{b-b_0}{2}\right)
+\left(\frac{b'-b_0}{2}\right)-\left(\frac{\widetilde{b}-b_0}{2}\right)
\Mod{N},
\end{equation*}
where $a'^{-1}$ is the inverse of $a'$ in $(\mathbb{Z}/N\mathbb{Z})^\times$.
If we let $T(\ell)=\begin{bmatrix}
1 & \ell\\0 & 1\end{bmatrix}
\in\Gamma_1(N)$, then we have
\begin{equation*}
F^{T(\ell)}=\widetilde{a}x^2+(\widetilde{b}+2\widetilde{a}\ell)xy
+(\widetilde{a}\ell^2+\widetilde{b}\ell+\widetilde{c})y^2. 
\end{equation*} 
It follows from Proposition \ref{isomorphism2} that
\begin{equation*}
\left(\sqrt[N]{\mathfrak{t}}\right)^{\psi_N([F^{T(\ell)}]_{\Gamma(N)})}=
\zeta_N^{a'^{-1}\left(\frac{b-b_0}{2}\right)+\left(\frac{b'-b_0}{2}\right)}\sqrt[N]{\mathfrak{t}}. 
\end{equation*}
Now, we define a binary operation on $C_{\Gamma(N)}(D,N)$ by letting
\begin{equation*}
\mathcal{C}\cdot\mathcal{C}'=[F^{T(\ell)}]_{\Gamma(N)}.
\end{equation*}
\end{definition}

\begin{remark}
Unlike Definition \ref{Noperation}, this Definition \ref{Noperation2}
does not appear in \cite{J-K-S-Y}. 
\end{remark}

\section {Adelic framed binary quadratic forms}

We introduce adelic framed binary quadratic forms and 
define the corresponding orbit space $\widehat{C}(D)$ under the action of
$\mathrm{SL}_2(\mathbb{Z})$. 
For each $N\in\mathbb{N}$, we define a natural map
$\rho_N:\widehat{C}(D)\rightarrow C_{\Gamma(N)}(D,\,N)$
which sends an adelic framed form class to its associated form class of level $N$.
\par
Define the set of adelic framed binary quadratic forms of discriminant $D$ by
\begin{equation*}
\widehat{\mathcal{Q}}(D)
=\left\{(Q,\,\gamma)\in\mathcal{Q}(D)\times\mathrm{SL}_2(\widehat{\mathbb{Z}})~|~
Q\left(\gamma\begin{bmatrix}1\\0\end{bmatrix}\right)\in
\widehat{\mathbb{Z}}^\times\right\}.
\end{equation*}
Note that $Q\left(\gamma\begin{bmatrix}1\\0\end{bmatrix}\right)$ is
the coefficient of $x^2$ in the binary quadratic form
$Q\left(\gamma\begin{bmatrix}x\\y\end{bmatrix}\right)$.
We regard $\mathrm{SL}_2(\mathbb{Z})$ as a subgroup of
$\mathrm{SL}_2(\widehat{\mathbb{Z}})$ via the diagonal embedding.

\begin{lemma}\label{action}
There is a right action of the modular group $\mathrm{SL}_2(\mathbb{Z})$ on 
the set $\widehat{\mathcal{Q}}(D)$ given by
\begin{equation}\label{rightaction}
(Q,\,\gamma)^\alpha=(Q^\alpha,\,\alpha^{-1}\gamma). 
\end{equation}
\end{lemma}
\begin{proof}
Let $(Q,\,\gamma)\in\widehat{\mathcal{Q}}(D)$ and $\alpha,\,\beta\in\mathrm{SL}_2(\mathbb{Z})$. 
We see that
 $(Q,\,\gamma)^\alpha=(Q^\alpha,\,\alpha^{-1}\gamma)$ again belongs to $\widehat{\mathcal{Q}}(D)$ because 
\begin{equation*}
Q^\alpha\left((\alpha^{-1}\gamma)\begin{bmatrix}1\\0\end{bmatrix}\right)=
Q\left(\alpha(\alpha^{-1}\gamma)\begin{bmatrix}1\\0\end{bmatrix}\right)
=Q\left(\gamma\begin{bmatrix}1\\0\end{bmatrix}\right)\in\widehat{\mathbb{Z}}^\times. 
\end{equation*}
It is straightforward that $(Q,\,\gamma)^{I_2}=(Q,\,\gamma)$. 
Furthermore, 
\begin{equation*}
((Q,\,\gamma)^\alpha)^\beta=(Q^\alpha,\,\alpha^{-1}\gamma)^\beta=
((Q^\alpha)^\beta,\,\beta^{-1}(\alpha^{-1}\gamma))=
(Q^{\alpha\beta},\,(\alpha\beta)^{-1}\gamma)=(Q,\,\gamma)^{\alpha\beta}. 
\end{equation*}
Therefore, \eqref{rightaction} defines a right action of
$\mathrm{SL}_2(\mathbb{Z})$ on $\widehat{\mathcal{Q}}(D)$.
\end{proof}

Let $\widehat{C}(D)$ be the orbit space of $\widehat{\mathcal{Q}}(D)$ 
with respect to the action of $\mathrm{SL}_2(\mathbb{Z})$ defined in Lemma \ref{action}.
Equivalently, $\widehat{C}(D)$ is the quotient of
$\widehat{\mathcal{Q}}(D)$ by the equivalence relation $\sim$ defined
as follows: for $(Q,\,\gamma),\,(Q',\,\gamma')\in\widehat{\mathcal{Q}}(D)$
\begin{equation*}
(Q,\,\gamma)\sim(Q',\,\gamma')\quad\Longleftrightarrow\quad
(Q',\,\gamma')=(Q,\,\gamma)^\alpha~\text{for some}~\alpha\in\mathrm{SL}_2(\mathbb{Z}).
\end{equation*}
If we denote the equivalence class of $(Q,\,\gamma)\in\widehat{\mathcal{Q}}(D)$ by $[(Q,\,\gamma)]$,
then we have
\begin{equation*}
\widehat{C}(D)=\widehat{\mathcal{Q}}(D)/\sim=\{[(Q,\,\gamma)]~|~(Q,\,\gamma)\in\widehat{\mathcal{Q}}(D)\}. 
\end{equation*}
\par
For $\gamma\in\mathrm{SL}_2(\widehat{\mathbb{Z}})$
and $N\in\mathbb{N}$, 
let $\widetilde{\gamma}_N$ denote any element of
$\mathrm{SL}_2(\mathbb{Z})$ satisfying
\begin{equation*}
\widetilde{\gamma}_N\equiv\gamma
\Mod{NM_2(\widehat{\mathbb{Z}})}.
\end{equation*}
Such a lift modulo $N$ always exists since the reduction map
\begin{equation*}
\pi_N:\mathrm{SL}_2(\mathbb{Z})
\rightarrow\mathrm{SL}_2(\mathbb{Z}/N\mathbb{Z})
\end{equation*} 
is a surjective homomorphism; see \cite[Lemma 1.38]{Shimura}. 
For $Q\in\mathcal{Q}(D,\,N)$, we simply write its class in 
$C_{\Gamma(N)}(D,\,N)=\mathcal{Q}(D,\,N)/\sim_{\Gamma(N)}$ as
$[Q]_N$, instead of $[Q]_{\Gamma(N)}$. 

\begin{lemma}\label{rNwell}
Let $(Q,\,\gamma)\in\widehat{\mathcal{Q}}(D)$. Then
$Q^{\widetilde{\gamma}_N}$ belongs to $\mathcal{Q}(D,\,N)$, and the
class $[Q^{\widetilde{\gamma}_N}]_N$ depends only on
$[(Q,\,\gamma)]$.
\end{lemma}
\begin{proof}
Since $Q\left(\gamma\begin{bmatrix}1\\0\end{bmatrix}\right)\in\widehat{\mathbb{Z}}^\times$
and $\widetilde{\gamma}_N\equiv\gamma\Mod{NM_2(\widehat{\mathbb{Z}})}$, 
the integer $Q\left(\widetilde{\gamma}_N\begin{bmatrix}1\\0\end{bmatrix}\right)$ is relatively prime to $N$.
Thus $Q^{\widetilde{\gamma}_N}$ belongs to $\mathcal{Q}(D,\,N)$. 
\par
Note that 
if
$\alpha,\,\beta\in\mathrm{SL}_2(\mathbb{Z})$ satisfy
\begin{equation*}
\alpha\equiv\beta\equiv\gamma\Mod{NM_2(\widehat{\mathbb{Z}})}, 
\end{equation*}
then $\alpha^{-1}\beta\in\Gamma(N)$ and so
\begin{equation*}
[Q^\alpha]_N=[(Q^\alpha)^{\alpha^{-1}\beta}]_N=[Q^\beta]_N.
\end{equation*}
This implies that $[Q^{\widetilde{\gamma}_N}]_N$ does not
depend on the choice of $\widetilde{\gamma}_N$.
\par
Suppose that $(Q',\,\gamma')\in\widehat{\mathcal{Q}}(D)$ satisfies
$(Q,\,\gamma)\sim(Q',\,\gamma')$. Then we have
\begin{equation*}
Q'=Q^\alpha~\text{and}~\gamma'=\alpha^{-1}\gamma~\text{for some}~
\alpha\in\mathrm{SL}_2(\mathbb{Z}). 
\end{equation*}
Observe that
$\widetilde{\gamma}'_N\equiv\alpha^{-1}\widetilde{\gamma}_N
\Mod{NM_2(\mathbb{Z})}$,
and hence
\begin{equation*}
[Q'^{\widetilde{\gamma}_N'}]_N=[(Q^\alpha)^{\alpha^{-1}\widetilde{\gamma}_N}]_N=[Q^{\widetilde{\gamma}_N}]_N.
\end{equation*}
This proves that $[Q^{\widetilde{\gamma}_N}]_N$ depends only on $[(Q,\,\gamma)]$. 
\end{proof}

\begin{definition}
For each $N\in\mathbb{N}$, we define a map
\begin{equation*}
\rho_N:\widehat{C}(D)\rightarrow C_{\Gamma(N)}(D,\,N),
\quad
\end{equation*}
by 
\begin{equation*}
\rho_N([(Q,\,\gamma)])=[Q^{\widetilde{\gamma}_N}]_N\quad((Q,\,\gamma)
\in\widehat{\mathcal{Q}}(D)).
\end{equation*}
By Lemma \ref{rNwell}, $\rho_N$ is well defined.
\end{definition}

\section {Inverse systems}\label{sect:inverse}

We organize the finite-level form class groups
$C_{\Gamma(N)}(D,\,N)$ into an inverse system and 
establish a lifting result used to prove the surjectivity of the inverse-limit map in Section \ref{onetoone}.
\par
For $M,\,N\in\mathbb{N}$ such that $N\,|\,M$, let 
\begin{equation*}
r_{M,\,N}:C_{\Gamma(M)}(D,\,M)\rightarrow
C_{\Gamma(N)}(D,\,N),\quad[Q]_M\mapsto[Q]_N\quad(Q\in\mathcal{Q}(D,\,M)). 
\end{equation*}
Since $\Gamma(M)\subseteq\Gamma(N)$, the map $r_{M,\,N}$ is well defined. 
By Lemma \ref{basic} (i) and Proposition
\ref{isomorphism2}, we get the following commutative diagram:
\begin{figure}[H]
\begin{equation*}
\xymatrixcolsep{5pc}
\xymatrix{
C_{\Gamma(M)}(D,\,M)\ar@{->}[r]^{\hspace{-0.5cm}\sim}_{\hspace{-0.3cm}\psi_M}
\ar@{->}[dd]_{r_{M,\,N}} 
 &\mathrm{Gal}\left(K_{\mathcal{O},\,M}(\sqrt[M]{\mathfrak{t}})/K(\mathfrak{t})\right) \ar@{->>}[dd]^{\textrm{restriction}} \\\\
C_{\Gamma(N)}(D,\,\,N) \ar@{->}[r]^{\hspace{-0.5cm}\sim}_{\hspace{-0.3cm}\psi_N} & 
\mathrm{Gal}\left(K_{\mathcal{O},\,N}(\sqrt[N]{\mathfrak{t}})/K(\mathfrak{t})\right)
}
\end{equation*}
\caption{A commutative diagram of groups of finite levels}
\label{diagram1}
\end{figure}
Thus  $r_{M,\,N}$ is a surjective homomorphism,
and we obtain an inverse system
\begin{equation*}
\{C_{\Gamma(N)}(D,\,N)\}_{N\in\mathbb{N}}
~\textrm{with transition maps $r_{M,\,N}$ for $N\,|\,M$}. 
\end{equation*}

\begin{definition}\label{TPR}
Let
\begin{equation*}
\mathcal{P}=(\mathcal{P}_N)_N\in\displaystyle\varprojlim_NC_{\Gamma(N)}(D,\,N). 
\end{equation*}
Fix a representative $R\in\mathcal{Q}(D)$ of $\mathcal{P}_1$. 
For each $N\in\mathbb{N}$, we define a subset
\begin{equation*}
T(\mathcal{P},\,R,\,N)=\left\{
\beta\in\mathrm{SL}_2(\mathbb{Z}/N\mathbb{Z})~|~
R^{\widetilde{\beta}}\in\mathcal{Q}(D,\,N)~
\text{and}~[R^{\widetilde{\beta}}]_N=\mathcal{P}_N\right\}
\end{equation*}
of $\mathrm{SL}_2(\mathbb{Z}/N\mathbb{Z})$, 
where $\widetilde{\beta}\in\mathrm{SL}_2(\mathbb{Z})$ is any lift of $\beta$. 
Note that the definition of $T(\mathcal{P},\,R,\,N)$ does not 
depend on the choice of $\widetilde{\beta}$. 
\end{definition}

For $M,\,N\in\mathbb{N}$ with $N\,|\,M$, we denote by
\begin{equation*}
\pi_{M,\,N}:\mathrm{SL}_2(\mathbb{Z}/M\mathbb{Z})
\rightarrow\mathrm{SL}_2(\mathbb{Z}/N\mathbb{Z})
\end{equation*}
the natural homomorphism induced from reduction. 

\begin{lemma}\label{TLemma}
With the notation of \textup{Definition \ref{TPR}}, let $M,\,N\in\mathbb{N}$.
\begin{enumerate}
\item[\textup{(i)}] We attain $T(\mathcal{P},\,R,\,N)\neq\emptyset$. 
\item[\textup{(ii)}] If $N\,|\,M$, then $\pi_{M,\,N}$ restricts to 
a well-defined surjection
$T(\mathcal{P},\,R,\,M)\rightarrow
T(\mathcal{P},\,R,\,N)$. 
\end{enumerate}
\end{lemma}
\begin{proof}
\begin{enumerate}
\item[(i)] Take a representative $Q\in\mathcal{Q}(D,\,N)$ of $\mathcal{P}_N$.
Since
\begin{equation*}
[R]_1=\mathcal{P}_1=r_{N,\,1}(\mathcal{P}_N)=r_{N,\,1}([Q]_N)=
[Q]_1,
\end{equation*}
we get
$Q=R^\alpha$ for some $\alpha\in\mathrm{SL}_2(\mathbb{Z})$ and so
\begin{equation*}
\mathcal{P}_N=[R^\alpha]_N. 
\end{equation*}
Thus we see that $\pi_N(\alpha)\in T(\mathcal{P},\,R,\,N)$,
and hence $T(\mathcal{P},\,R,\,N)\neq\emptyset$. 
\item[(ii)] Let $\gamma\in T(\mathcal{P},\,R,\,M)$. Then we have $R^{\widetilde{\gamma}}\in\mathcal{Q}(D,\,M)$
and
\begin{equation*}
[R^{\widetilde{\gamma}}]_M=\mathcal{P}_M,
\end{equation*}
where $\widetilde{\gamma}\in\mathrm{SL}_2(\mathbb{Z})$ is a lift of $\gamma$. 
Applying $r_{M,\,N}$ to both sides gives 
\begin{equation*}
[R^{\widetilde{\gamma}}]_N=\mathcal{P}_N. 
\end{equation*}
This shows that 
$\pi_{M,\,N}(\gamma)=\pi_N(\widetilde{\gamma})\in T(\mathcal{P},\,R,\,N)$, and then
the reduction $\pi_{M,\,N}$ induces a well-defined map 
$T(\mathcal{P},\,R,\,M)\rightarrow
T(\mathcal{P},\,R,\,N)$.
\par
To show that  $T(\mathcal{P},\,R,\,M)\rightarrow T(\mathcal{P},\,R,\,N)$ is surjective, 
let $\beta\in T(\mathcal{P},\,R,\,N)$. 
Choose a lift $\widetilde{\beta}\in\mathrm{SL}_2(\mathbb{Z})$ of $\beta$. Then we have
$R^{\widetilde{\beta}}\in\mathcal{Q}(D,\,N)$ and
\begin{equation}\label{RPm}
[R^{\widetilde{\beta}}]_N=\mathcal{P}_N. 
\end{equation}
Take an element 
$\delta\in T(\mathcal{P},\,R,\,M)$ with a lift $\widetilde{\delta}\in\mathrm{SL}_2(\mathbb{Z})$, so
$R^{\widetilde{\delta}}\in\mathcal{Q}(D,\,M)$ and
\begin{equation}\label{RdP}
[R^{\widetilde{\delta}}]_M=\mathcal{P}_M. 
\end{equation} 
Applying $r_{M,\,N}$ to \eqref{RdP} and using \eqref{RPm}, we obtain
\begin{equation*}
[R^{\widetilde{\delta}}]_N=[R^{\widetilde{\beta}}]_N.
\end{equation*}
This implies that
\begin{equation*}
R^{\widetilde{\beta}}=(R^{\widetilde{\delta}})^\alpha\quad\text{for some}~
\alpha\in\Gamma(N),
\end{equation*}
from which we get 
\begin{equation}\label{RRR}
R^{\widetilde{\beta}\alpha^{-1}}=R^{\widetilde{\delta}}\in\mathcal{Q}(D,\,M).
\end{equation}
And, we achieve by \eqref{RdP} and \eqref{RRR} that
\begin{equation*}
\pi_M(\widetilde{\beta}\alpha^{-1})\in T(\mathcal{P},\,R,\,M). 
\end{equation*}
Furthermore, since $\alpha\in\Gamma(N)$, we derive
\begin{equation*}
\pi_N(\widetilde{\beta}\alpha^{-1})=\pi_N(\widetilde{\beta})=\beta. 
\end{equation*}
Therefore the map $T(\mathcal{P},\,R,\,M)\rightarrow T(\mathcal{P},\,R,\,N)$
is surjective. 
\end{enumerate}
\end{proof}

By Lemma \ref{TLemma}, we obtain an inverse system
\begin{equation*}
\{T(\mathcal{P},\,R,\,N)\}_{N\in\mathbb{N}}
\end{equation*}
of nonempty finite sets whose transition maps
are the surjections arising from $\pi_{M,\,N}$
for $N\,|\,M$. 

\section {Binary operation on $\widehat{C}(D)$}\label{binaryoperation}

We define a binary operation on $\widehat{C}(D)$ 
by combining the Gauss-Dirichlet composition of the two quadratic forms 
with an explicit construction of the corresponding adelic frame. 
The well-definedness of this operation will be proved in Section \ref{mainsection}.

\begin{definition}\label{explicit}
Let
\begin{align*}
\mathcal{C}&=[(Q,\,\gamma)]~\textrm{with}~(Q,\,\gamma)\in\widehat{\mathcal{Q}}(D),\\
\mathcal{C}'&=[(Q',\,\gamma')]~\textrm{with}~(Q',\,\gamma')\in\widehat{\mathcal{Q}}(D).
\end{align*}
Associated with $\mathcal{C}$ and $\mathcal{C}'$,
we construct an element $(R,\,\nu)\in\widehat{\mathcal{Q}}(D)$ through the following steps:
\begin{enumerate}
\item[Step 1.] Write
\begin{equation*}
Q_1=Q=a_1x^2+b_1xy+c_1y^2. 
\end{equation*}
Let $\alpha$ be an element of $\mathrm{SL}_2(\mathbb{Z})$ 
so that the form
\begin{equation*}
Q_2=Q'^\alpha=a_2x^2+b_2xy+c_2y^2
\end{equation*}
satisfies
\begin{equation*}
\gcd\left(a_1,\,a_2,\frac{b_1+b_2}{2}\right)=1. 
\end{equation*}
Furthermore, let $B$ be an integer satisfying
\begin{align*}
B\equiv b_1\Mod{2a_1},\quad
B\equiv b_2\Mod{2a_2},\quad
B^2\equiv D\Mod{4a_1a_2}.\nonumber
\end{align*}
Define a Dirichlet composition $R$ of $Q_1$ and $Q_2$ by
\begin{equation*}
R=Ax^2+Bxy+Cy^2~\textrm{with}~
A=a_1a_2~\textrm{and}~C=\frac{B^2-D}{4a_1a_2}. 
\end{equation*}
If we set
\begin{equation*}
\beta_1=\gamma\quad\textrm{and}\quad
\beta_2=\alpha^{-1}\gamma'
\end{equation*}
then we see that
\begin{equation}\label{CQbCQb}
\mathcal{C}=[(Q_1,\,\beta_1)]\quad\textrm{and}\quad
\mathcal{C}'=[(Q_2,\,\beta_2)].
\end{equation}
\item[Step 2.] Write
\begin{equation}\label{beta}
\beta_i^{-1}=\begin{bmatrix}
u_i&v_i\\r_i&s_i\end{bmatrix}
\in\mathrm{SL}_2(\widehat{\mathbb{Z}})\quad(i=1,\,2).  
\end{equation}
And set
\begin{equation*}
\mu_i=s_i+r_i\frac{B-b_i}{2a_i}
\in\widehat{\mathbb{Z}}\quad (i=1,2).
\end{equation*}
Since
\begin{equation*}
\omega_{Q_1}=\frac{B-b_1}{2a_1}+a_2\omega_R
\quad\textrm{and}\quad
\omega_{Q_2}=\frac{B-b_2}{2a_2}+a_1\omega_R,
\end{equation*}
we have
\begin{equation}\label{srma}
s_1+r_1\omega_{Q_1}=\mu_1+a_2r_1\omega_R
\quad\textrm{and}\quad
s_2+r_2\omega_{Q_2}=\mu_2+a_1r_2\omega_R.
\end{equation}
Define $r,\,s\in\widehat{\mathbb{Z}}$ by 
\begin{equation}\label{rands}
r=-Br_1r_2+a_2r_1\mu_2+a_1r_2\mu_1
\quad\textrm{and}\quad
s=\mu_1\mu_2-Cr_1r_2.
\end{equation}
Then one can derive by (\ref{srma}) and the relation 
$A\omega_R^2+B\omega_R+C=0$ that
\begin{equation}\label{srsrsr1}
(s_1+r_1\omega_{Q_1})(s_2+r_2\omega_{Q_2})=s+r\omega_R,
\end{equation}
and similarly 
\begin{equation}\label{srsrsr2}
(s_1+r_1\overline{\omega}_{Q_1})(s_2+r_2\overline{\omega}_{Q_2})=s+r\overline{\omega}_R.
\end{equation}
Write
\begin{equation}\label{Qbhat}
Q_i\left(\beta_i\begin{bmatrix}x\\y\end{bmatrix}\right)=
\widehat{a}_ix^2+\widehat{b}_ixy+\widehat{c}_iy^2
\in\widehat{\mathbb{Z}}[x,y]\quad(i=1,\,2).
\end{equation}
Since $(Q_i,\,\beta_i)\in\widehat{\mathcal{Q}}(D)$, we get
\begin{equation}\label{ai}
\widehat{a}_i\in\widehat{\mathbb{Z}}^\times\quad(i=1,\,2).
\end{equation}
On the other hand, we attain by \eqref{beta} that
\begin{equation*}
\widehat{a}_i=
Q_i(s_i,\,-r_i)\quad (i=1,\,2).
\end{equation*}
It follows from \eqref{srsrsr1}--\eqref{ai} that
\begin{equation}\label{unit}
R(s,\,-r)=\widehat{a}_1\widehat{a}_2\in\widehat{\mathbb{Z}}^\times. 
\end{equation}
Put
\begin{equation*}
\widehat{a}=\widehat{a}_1\widehat{a}_2\in\widehat{\mathbb{Z}}^\times.
\end{equation*}
\item[Step 3.]
By \eqref{unit}, there exist $u_0,\,v_0\in\widehat{\mathbb{Z}}$ such that
\begin{equation}\label{1}
u_0s-v_0r=1.
\end{equation}
For instance, one can take
\begin{equation*}
u_0=\widehat{a}^{-1}(As-Br)\quad
\textrm{and}\quad v_0=-\widehat{a}^{-1}Cr. 
\end{equation*}
Let $m$ be the coefficient of $xy$ in 
\begin{equation*}
R\left(\begin{bmatrix}u_0&v_0\\r&s\end{bmatrix}^{-1}\begin{bmatrix}x\\y\end{bmatrix}\right)
=R\left(\begin{bmatrix}s&-v_0\\-r&u_0\end{bmatrix}\begin{bmatrix}x\\y\end{bmatrix}\right),
\end{equation*}
namely,
\begin{equation}\label{m}
m=-2Av_0s+B(u_0s+v_0r)-2Cu_0r.
\end{equation}
Furthermore, let
\begin{equation}\label{bb2a2}
\widehat{b}=b_0+2\widehat{a}_2^{-1}\left(\frac{\widehat{b}_1-b_0}{2}\right)
+2\left(\frac{\widehat{b}_2-b_0}{2}\right),
\end{equation}
which lies in $\widehat{\mathbb{Z}}$
since the identity
$\widehat{b}_i^2-4\widehat{a}_i\widehat{c}_i=D=b_0^2-4a_0c_0$ 
leads to $\displaystyle\frac{\widehat{b}_i-b_0}{2}\in\widehat{\mathbb{Z}}$ ($i=1,\,2$).
Hence we achieve that
\begin{align*}
m&\equiv B(u_0s+v_0r)\Mod{2\widehat{\mathbb{Z}}}\quad
\textrm{by \eqref{m}}\\
&\equiv B\Mod{2\widehat{\mathbb{Z}}}\quad\textrm{by \eqref{1}}\\
&\equiv b_0\Mod{2\widehat{\mathbb{Z}}}\quad\textrm{because}~
B^2-4AC=D=b_0^2-4a_0c_0\\
&\equiv\widehat{b}\Mod{2\widehat{\mathbb{Z}}}\quad\textrm{due to}~
\frac{\widehat{b}_i-b_0}{2}\in\widehat{\mathbb{Z}}~(i=1,\,2). 
\end{align*}
Combining this with \eqref{ai}, we obtain
\begin{equation*}
k:=\frac{m-\widehat{b}}{2\widehat{a}}\in\widehat{\mathbb{Z}}.
\end{equation*}
Now, put
\begin{equation*}
u=u_0+kr\quad\textrm{and}\quad
v=v_0+ks,
\end{equation*}
and define
\begin{equation*}
\nu=\begin{bmatrix}u&v\\r&s\end{bmatrix}^{-1}
=\begin{bmatrix}s&-v\\-r&u\end{bmatrix}
\in\mathrm{SL}_2(\widehat{\mathbb{Z}}).
\end{equation*}
Then we have by construction 
\begin{equation*}
R\left(\nu\begin{bmatrix}x\\y\end{bmatrix}\right)=
R\left(\begin{bmatrix}s&-v_0\\-r&u_0\end{bmatrix}\begin{bmatrix}1&-k\\0&1\end{bmatrix}\begin{bmatrix}x\\y\end{bmatrix}
\right)=\widehat{a}x^2
+\widehat{b}xy+\widehat{c}y^2~\textrm{with}~
\widehat{c}=\frac{\widehat{b}^2-D}{4\widehat{a}}
\in\widehat{\mathbb{Z}}.
\end{equation*}
\end{enumerate}
For the chosen data, we set
\begin{equation*}
\mathcal{C}\cdot\mathcal{C}'=[(R,\,\nu)].
\end{equation*}
\end{definition}

\section {One-to-one correspondence between $\widehat{C}(D)$
and $\protect\varprojlim_N C_{\Gamma(N)}(D,\,N)$}\label{onetoone}

We establish a one-to-one correspondence between adelic framed form classes of discriminant $D$
and compatible systems of finite-level form classes.

\begin{lemma}\label{Cwelldefined}
If $\mathcal{C}\in\widehat{C}(D)$, then
$(\rho_N(\mathcal{C}))_N\in\varprojlim_N C_{\Gamma(N)}(D,\,N)$. 
\end{lemma}
\begin{proof}
We get $\mathcal{C}=[(Q,\,\gamma)]$ for some $(Q,\,\gamma)\in\widehat{\mathcal{Q}}(D)$. 
Let $M,\,N\in\mathbb{N}$ satisfy $N\,|\,M$. 
We derive that
\begin{equation*}
r_{M,\,N}(\rho_M(\mathcal{C}))=r_{M,\,N}([Q^{\widetilde{\gamma}_M}]_M)
=[Q^{\widetilde{\gamma}_M}]_N\\
=[Q^{\widetilde{\gamma}_N}]_N=\rho_N(\mathcal{C})
\end{equation*}
because $\widetilde{\gamma}_M\equiv\widetilde{\gamma}_N\Mod{NM_2(\mathbb{Z})}$. 
And, $(\rho_N(\mathcal{C}))_N$ belongs to $\varprojlim_N C_{\Gamma(N)}(D,\,N)$.
\end{proof}

\begin{definition}\label{rDef}
Using Lemma \ref{Cwelldefined}, we define a map
\begin{equation*}
\rho:\widehat{C}(D)\rightarrow
\varprojlim_N C_{\Gamma(N)}(D,\,N),\quad\mathcal{C}\mapsto(\rho_N(\mathcal{C}))_N.
\end{equation*}
\end{definition}

\begin{lemma}\label{finite}
Let $T$ be a finite subset of $\mathrm{SL}_2(\widehat{\mathbb{Z}})$, and let $\nu\in\mathrm{SL}_2(\widehat{\mathbb{Z}})$.
Suppose that
for each $N\in\mathbb{N}$ there exists $t_N\in T$ such that
\begin{equation}\label{ns}
\nu\equiv t_N\Mod{NM_2(\widehat{\mathbb{Z}})}.
\end{equation}
Then $\nu$ belongs to $T$. 
\end{lemma}
\begin{proof}
By the assumption \eqref{ns} 
applied to $N=m!$ ($m\in\mathbb{N}$)
and the fact $T$ is a finite subset of $\mathrm{SL}_2(\widehat{\mathbb{Z}})$, 
there exist a strictly increasing sequence $m_1<m_2<m_3<\cdots$ of positive integers
and an element $t\in T$ such that
\begin{equation}\label{nsmM}
\nu\equiv t\Mod{m_j!M_2(\widehat{\mathbb{Z}})}~\text{for all}~j\in\mathbb{N}.  
\end{equation}
For each $n\in\mathbb{N}$, we see that
$n\,|\,m_j!$ for sufficiently large $j$, and so
$m_j!M_2(\widehat{\mathbb{Z}})\subseteq
nM_2(\widehat{\mathbb{Z}})$. 
Hence we get by \eqref{nsmM} that
\begin{equation*}
\nu-t\in\bigcap_{n\geq1}nM_2(\widehat{\mathbb{Z}})=\{O_2\}
\end{equation*}
due to the fact $\bigcap_{n\geq1}n\widehat{\mathbb{Z}}=\{0\}$. It then follows that
$\nu=t\in T$. 
\end{proof}

\begin{remark}\label{alternative}
The group $\mathrm{SL}_2(\widehat{\mathbb{Z}})$ carries its natural profinite topology 
through the canonical isomorphism
\begin{equation*}
\mathrm{SL}_2(\widehat{\mathbb{Z}})\simeq\varprojlim_{N\in\mathbb{N}}\mathrm{SL}_2(\mathbb{Z}/N\mathbb{Z}),
\end{equation*}
where $\mathbb{N}$ is directed by divisibility and each 
finite group $\mathrm{SL}_2(\mathbb{Z}/N\mathbb{Z})$
is given the discrete topology. 
Equivalently, this topology agrees with the subspace topology 
inherited from $M_2(\widehat{\mathbb{Z}})\simeq\widehat{\mathbb{Z}}^4$. 
Put
\begin{equation*}
\mathfrak{K}(N)=\mathrm{ker}\big(\mathrm{SL}_2(\widehat{\mathbb{Z}})
\rightarrow\mathrm{SL}_2(\mathbb{Z}/N\mathbb{Z})\big)
=(I_2+NM_2(\widehat{\mathbb{Z}}))\cap\mathrm{SL}_2(\widehat{\mathbb{Z}})
\quad(N\in\mathbb{N}).
\end{equation*}
Then the subgroups $\mathfrak{K}(N)$ form a neighborhood basis of 
the identity $I_2\in\mathrm{SL}_2(\widehat{\mathbb{Z}})$,
and the cosets $\nu\mathfrak{K}(N)$ form a neighborhood basis of $\nu\in
\mathrm{SL}_2(\widehat{\mathbb{Z}})$. 
The condition \eqref{ns} in Lemma \ref{finite} yields that 
$t_N\in\nu\mathfrak{K}(N)$ for each $N\in\mathbb{N}$,
and hence every neighborhood of $\nu$ intersects $T$. 
Moreover, since $\mathrm{SL}_2(\widehat{\mathbb{Z}})$ is Hausdorff,
the finite subset $T$ is closed in $\mathrm{SL}_2(\widehat{\mathbb{Z}})$.
Therefore $\nu$ belongs to $T$. This also gives an alternative
topological proof of Lemma \ref{finite}.
\end{remark}

\begin{proposition}\label{bijective}
The map $\rho$ defined in Definition \ref{rDef} is bijective. 
\end{proposition}
\begin{proof}
To show the injectivity of $\rho$, suppose that
\begin{equation}\label{rCrC'CC'}
\rho(\mathcal{C})=\rho(\mathcal{C}')
\quad\textrm{for some}~\mathcal{C},\,\mathcal{C}'\in\widehat{C}(D). 
\end{equation}
If
\begin{equation*}
\mathcal{C}=[(Q,\,\gamma)]~\text{and}~\mathcal{C}'=[(Q',\,\gamma')]~
\text{with}~(Q,\,\gamma),\,(Q',\,\gamma')\in\widehat{\mathcal{Q}}(D), 
\end{equation*}
then we get by \eqref{rCrC'CC'} that
\begin{equation*}
[Q^{\widetilde{\gamma}_N}]_N=
[Q'^{\widetilde{\gamma'}_N}]_N~\text{for all}~N\in\mathbb{N}. 
\end{equation*}
Since $[Q]_1=[Q^{\widetilde{\gamma}_1}]_1=[Q'^{\widetilde{\gamma'}_1}]_1=[Q']_1$,
we have
\begin{equation*}
Q'=Q^\alpha~\text{for some}~\alpha\in\mathrm{SL}_2(\mathbb{Z}). 
\end{equation*}
Then we deduce that in $\widehat{C}(D)$
\begin{equation}\label{QQQQ}
[(Q',\,\gamma')]=[(Q',\,\gamma')^{\alpha^{-1}}]=
[(Q'^{\alpha^{-1}},\,\alpha\gamma')]=
[(Q,\,\gamma'')]~\text{with}~\gamma''=\alpha\gamma'.
\end{equation}
Therefore we obtain by \eqref{rCrC'CC'} that in $\varprojlim_N C_{\Gamma(N)}(D,\,N)$
\begin{equation*}
([Q^{\widetilde{\gamma}_N}]_N)_N=
([Q^{\widetilde{\gamma}''_N}]_N)_N.
\end{equation*}
Hence, for each $N\in\mathbb{N}$ there is
an element $\alpha_N\in\Gamma(N)$ satisfying
$Q^{\widetilde{\gamma}''_N}=(Q^{\widetilde{\gamma}_N})^{\alpha_N}$, and so
\begin{equation*}
Q^{\widetilde{\gamma}''_N\alpha_N^{-1}\widetilde{\gamma}_N^{-1}}=Q. 
\end{equation*}
If we set $\beta_N=\widetilde{\gamma}''_N\alpha_N^{-1}\widetilde{\gamma}_N^{-1}$,
then we achieve
\begin{equation}\label{bS}
\beta_N\in\mathrm{Stab}(Q),
\end{equation}
where $\mathrm{Stab}(Q)$ denotes the stabilizer subgroup of $Q$ in $\mathrm{SL}_2(\mathbb{Z})$. 
Since $\alpha_N\in\Gamma(N)$, we have
\begin{equation}\label{beta_N}
\gamma''\gamma^{-1}\equiv\beta_N\Mod{NM_2(\widehat{\mathbb{Z}})}
\quad\text{for every}~N\in\mathbb{N}. 
\end{equation}
Since $\mathrm{Stab}(Q)$ is a finite set (see 
Remark \ref{QQzz} (ii) and \cite[Proposition 1.5 in Chapter I]{Silverman}),
applying Lemma \ref{finite} to \eqref{bS} and \eqref{beta_N} yields that
\begin{equation}\label{gbg}
\gamma''=\beta\gamma\quad\text{for some}~\beta\in\mathrm{Stab}(Q). 
\end{equation} 
And, we derive that in $\widehat{C}(D)$
\begin{align*}
\mathcal{C}'&=[(Q,\,\gamma'')]\quad\text{by \eqref{QQQQ}}\\
&=[(Q,\,\beta\gamma)]\quad\text{by \eqref{gbg}}\\
&=[(Q^{\beta^{-1}},\,\beta\gamma)]\quad\text{since}~\beta\in\mathrm{Stab}(Q)\\
&=[(Q,\,\gamma)^{\beta^{-1}}]\\
&=\mathcal{C}. 
\end{align*}
Thus $\rho$ is injective.
\par
To prove the surjectivity of $\rho$, let 
\begin{equation*}
\mathcal{P}=(\mathcal{P}_N)_N\in\varprojlim_N C_{\Gamma(N)}(D,\,N)\quad\text{with}~
\mathcal{P}_N\in C_{\Gamma(N)}(D,\,N). 
\end{equation*}
Choose a representative $R\in\mathcal{Q}(D)$ of $\mathcal{P}_1$.
Since the family $\{T(\mathcal{P},\,R,\,N)\}_N$
is an inverse system of nonempty finite sets,
its inverse limit is nonempty (see \cite[Theorem 111]{I-M}). 
Take an element 
\begin{equation*}
\mu=(\mu_N)_N\in\varprojlim_N T(\mathcal{P},\,R,\,N)\quad(\subseteq\varprojlim_N\mathrm{SL}_2(\mathbb{Z}/N\mathbb{Z})).
\end{equation*}
For each $N\in\mathbb{N}$, let $\widetilde{\mu}_N\in\mathrm{SL}_2(\mathbb{Z})$ be
a lift of $\mu_N$. 
Let $\nu$ be the element of $\mathrm{SL}_2(\widehat{\mathbb{Z}})$ 
corresponding to $\mu$, and take $\widetilde{\nu}_N=\widetilde{\mu}_N$. 
Then, since
\begin{equation*}
R^{\widetilde{\nu}_N}
=R^{\widetilde{\mu}_N}\in\mathcal{Q}(D,\,N)\quad(N\in\mathbb{N}),
\end{equation*}
we attain 
\begin{equation*}
R\left(\nu\begin{bmatrix}1\\0\end{bmatrix}\right)\in\widehat{\mathbb{Z}}^\times, 
\end{equation*}
and so $(R,\,\nu)\in\widehat{\mathcal{Q}}(D)$. We then find that
\begin{equation*}
\rho([(R,\,\nu)])=(\rho_N([(R,\,\nu)]))_N
=([R^{\widetilde{\nu}_N}]_N)_N
=([R^{\widetilde{\mu}_N}]_N)_N
=(\mathcal{P}_N)_N
=\mathcal{P}. 
\end{equation*}
This proves that $\rho$ is surjective. 
\end{proof}

\section {The adelic framed form class group of discriminant $D$}\label{mainsection}

Now, we prove that the binary operation introduced in Section \ref{binaryoperation} is well defined and makes $\widehat{C}(D)$ into a group isomorphic to
$\mathrm{Gal}\left(K^\mathrm{ab}(\mathfrak{t}^{1/\infty})/K(\mathfrak{t})\right)$.

\begin{lemma}\label{compatible}
With the notation of Definition \ref{explicit}, we have
\begin{equation*}
[Q_1^{(\widetilde{\beta}_1)_N}]_{\Gamma_1(N)}
[Q_2^{(\widetilde{\beta}_2)_N}]_{\Gamma_1(N)}
=[R^{\widetilde{\nu}_N}]_{\Gamma_1(N)}
~\textrm{for every}~N\in\mathbb{N}.
\end{equation*}
\end{lemma}
\begin{proof}
Let $N\in\mathbb N$, and put
\begin{equation}\label{FQGR}
 F_i=Q_i^{(\widetilde{\beta}_i)_N}~
(i=1,2)\quad\textrm{and}\quad
 G=R^{\widetilde{\nu}_N}.
\end{equation}
For $Q\in\mathcal{Q}(D,\,N)$, we recall
\begin{equation*}
\mathfrak{a}_Q:=\mathbb Z\omega_Q+\mathbb Z\in I(\mathcal{O},\,N). 
\end{equation*}
By Proposition~3.4, it suffices to prove that in $C(\mathcal{O},\,N)=I(\mathcal{O},\,N)/P_1(\mathcal{O},\,N)$
\begin{equation*}
[\mathfrak{a}_{F_1}\mathfrak a_{F_2}]=
[\mathfrak{a}_G].
\end{equation*}
Write
\begin{equation*}
(\widetilde{\beta}_i)_N^{-1}
=\begin{bmatrix}
u_{i,\,N} & v_{i,\,N}\\
r_{i,\,N} & s_{i,\,N}
\end{bmatrix}~(i=1,\,2)
\quad\textrm{and}\quad
\widetilde{\nu}_N^{-1}
=\begin{bmatrix}u_N & v_N\\r_N & s_N
\end{bmatrix},
\end{equation*}
and set
\begin{equation*}
d_{i,\,N}=s_{i,\,N}+r_{i,\,N}\omega_{Q_i}
~(i=1,\,2)\quad\textrm{and}\quad d_N=s_N+r_N\omega_R.
\end{equation*}
We find by \eqref{FQGR} that
\begin{equation*}
\omega_{F_i}=\omega_{Q_i^{(\widetilde{\beta}_i)_N}}=
(\widetilde{\beta}_i)_N^{-1}(\omega_{Q_i})=
\frac{u_{i,\,N}\omega_{Q_i}+v_{i,\,N}}
{r_{i,\,N}\omega_{Q_i}+s_{i,\,N}}=
\frac{u_{i,\,N}\omega_{Q_i}+v_{i,\,N}}
{d_{i,\,N}}
 \quad (i=1,2)
\end{equation*}
and
\begin{equation*}
\omega_G=\omega_{R^{\widetilde{\nu}_N}}=
\widetilde{\nu}_N^{-1}(\omega_R)=
\frac{u_N\omega_R+v_N}{r_N\omega_R+s_N}=
\frac{u_N\omega_R+v_N}{d_N}.
\end{equation*}
It follows from the fact 
$(\widetilde{\beta}_i)_N^{-1},\,\widetilde{\nu}_N^{-1}
\in\mathrm{SL}_2(\mathbb{Z})$ that
\begin{align}
\mathfrak a_{F_i}&=
d_{i,\,N}^{-1}\mathfrak{a}_{Q_i}
\quad(i=1,\,2),\label{ada}\\
\mathfrak{a}_G&=d_N^{-1}\mathfrak{a}_R.\label{ada2}
\end{align}
Recall the notations \eqref{beta} and \eqref{rands}. If we let
\begin{equation*}
d_i=s_i+r_i\omega_{Q_i}
~(i=1,\,2)\quad\textrm{and}\quad
d=s+r\omega_R,
\end{equation*}
then \eqref{srsrsr1} gives
\begin{equation}\label{ddd}
d_1d_2=d.
\end{equation}
Since $R$ is a Dirichlet composition of $Q_1$ and $Q_2$, 
we get
\begin{equation}\label{aQaQaR}
\mathfrak{a}_{Q_1}\mathfrak{a}_{Q_2}=\mathfrak{a}_R
\end{equation}
(see \cite[(7.13)]{Cox}). 
Put
\begin{equation*}
\lambda_N=\frac{d_N}{d_{1,\,N}d_{2,\,N}}\quad(\in K).
\end{equation*}
Then we attain by \eqref{ada}--\eqref{aQaQaR} that
\begin{equation}\label{lOaaa}
\lambda_N\mathcal{O}=\mathfrak{a}_{F_1}\mathfrak{a}_{F_2}\mathfrak{a}_G^{-1}.
\end{equation}
Thus $\lambda_N\mathcal{O}$ belongs to $I(\mathcal{O},\,N)$, and hence
there exists a positive integer $m$ satisfying
\begin{equation}\label{mlm1N}
\gcd(m,\,N)=1\quad\textrm{and}\quad m\lambda_N\in\mathcal{O}.
\end{equation}
By replacing $m$ with 
a suitable positive power, we may further assume that
\begin{equation}\label{mlm1N2}
m\equiv1\Mod{N}.
\end{equation}
\par
Let $p$ be a prime divisor of $N$, and put
\begin{equation*}
K_p=K\otimes_\mathbb{Q}\mathbb{Q}_p\quad\textrm{and}
\quad
\mathcal{O}_p=\mathcal{O}\otimes_{\mathbb{Z}}\mathbb{Z}_p.
\end{equation*}
Since $\mathfrak{a}_{F_1},\,\mathfrak{a}_{F_2},\,\mathfrak{a}_G\in
I(\mathcal{O},\,N)$, one sees that in $K_p$
\begin{equation*}
\mathfrak a_{F_i}\otimes_{\mathbb{Z}}\mathbb{Z}_p
=\mathcal{O}_p~(i=1,\,2)
\quad\textrm{and}\quad
\mathfrak{a}_G\otimes_{\mathbb{Z}}\mathbb{Z}_p
=\mathcal{O}_p.
\end{equation*}
It then follows from \eqref{ada} and \eqref{ada2} that
\begin{align}
\mathfrak{a}_{Q_i}\otimes_{\mathbb{Z}}\mathbb{Z}_p
&=d_{i,\,N}\mathcal{O}_p\quad(i=1,\,2),\label{aZdO}\\
\mathfrak{a}_R\otimes_{\mathbb{Z}}\mathbb{Z}_p
 &=d_N\mathcal{O}_p.\label{aZdO2}
\end{align}
Now, we achieve that for $i=1,\,2$
\begin{align*}
d_i-d_{i,\,N}&=(s_i-s_{i,\,N})+(r_i-r_{i,\,N})\omega_{Q_i}\\
&\in N\left(\mathfrak{a}_{Q_i}\otimes_{\mathbb{Z}}\mathbb{Z}_p\right)
\quad\textrm{because}~(\widetilde{\beta}_i)_N^{-1}\equiv\beta_i^{-1}
\Mod{NM_2(\widehat{\mathbb{Z}})}\\
&=Nd_{i,\,N}\mathcal{O}_p\quad\textrm{by \eqref{aZdO}},
\end{align*}
and so
\begin{equation}\label{dd1N}
 \frac{d_i}{d_{i,\,N}}
 \in 1+N\mathcal{O}_p\quad (i=1,\,2).
\end{equation}
In a similar way, one can deduce from the fact 
$\widetilde{\nu}_N^{-1}\equiv\nu^{-1}\Mod{NM_2(\widehat{\mathbb{Z}})}$ and 
\eqref{aZdO2} that
\begin{equation}\label{dd1N2}
 \frac{d}{d_N}\in 1+N\mathcal{O}_p.
\end{equation}
Therefore we obtain by \eqref{ddd}, \eqref{dd1N} and \eqref{dd1N2} that
\begin{equation*}
\lambda_N=\frac{d_N}{d}\cdot\frac{d_1}{d_{1,\,N}}
\cdot\frac{d_2}{d_{2,\,N}}\in 1+N\mathcal{O}_p.
\end{equation*}
Here, since $p$ is an arbitrary prime dividing $N$, we get by \eqref{mlm1N}, \eqref{mlm1N2} and the identity
$\mathcal{O}_p=\mathbb{Z}_p\tau_\mathcal{O}+\mathbb{Z}_p$ that
\begin{equation}\label{ml1}
 m\lambda_N\equiv1\Mod{N\mathcal{O}}.
\end{equation}
\par
Finally, in $C(\mathcal{O},\,N)=I(\mathcal{O},\,N)/P_1(\mathcal{O},\,N)$ we have
\begin{align*}
[\mathfrak{a}_{F_1}\mathfrak{a}_{F_2}]
&=[(\lambda_N\mathcal{O})\mathfrak{a}_G]\quad\textrm{by \eqref{lOaaa}}\\
&=[(m\mathcal{O})^{-1}(m\lambda_N\mathcal{O})\mathfrak{a}_G]\\
&=[\mathfrak{a}_G]
\quad\textrm{since $m\mathcal{O}$  and $m\lambda_N\mathcal{O}$ belong to $P_1(\mathcal{O},\,N)$ by \eqref{mlm1N}, \eqref{mlm1N2} and \eqref{ml1}}. 
\end{align*}
\end{proof}

\begin{lemma}\label{mainlemma}
With the notation of Definition \ref{explicit}, 
we obtain
\begin{equation*}
\rho_N(\mathcal{C})\rho_N(\mathcal{C}')=
[R^{\widetilde{\nu}_N}]_N\quad\textrm{for every}~N\in\mathbb{N}.
\end{equation*}
\end{lemma}
\begin{proof}
Let $N\in\mathbb{N}$. We get by \eqref{CQbCQb} that
\begin{equation}\label{rQNrQN}
\rho_N(\mathcal{C})=[Q_1^{(\widetilde{\beta}_1)_N}]_N
\quad\textrm{and}\quad\rho_N(\mathcal{C}')=[Q_2^{(\widetilde{\beta}_2)_N}]_N.
\end{equation}
Write
\begin{equation}\label{RCRCR}
Q_1^{(\widetilde{\beta}_1)_N}=ax^2+bxy+cy^2,\quad
Q_2^{(\widetilde{\beta}_2)_N}=a'x^2+b'xy+c'y^2,\quad
R^{\widetilde{\nu}_N}=\widetilde{a}x^2+\widetilde{b}xy+\widetilde{c}y^2. 
\end{equation}
Let
\begin{equation*}
\sigma_N:C_{\Gamma_1(N)}(D,\,N)\stackrel{\sim}{\rightarrow}
\mathrm{Gal}(K_{\mathcal{O},\,N}/K)
\end{equation*}
and
\begin{equation*}
\psi_N:C_{\Gamma(N)}(D,\,N)\stackrel{\sim}{\rightarrow}
\mathrm{Gal}\left(K_{\mathcal{O},\,N}(\sqrt[N]{\mathfrak{t}})/
K(\mathfrak{t})\right)
\end{equation*}
denote the isomorphisms stated in Propositions \ref{isomorphism} and \ref{isomorphism2}, respectively. 
Then we derive that for any $f\in\mathcal{F}_{N,\,\tau_\mathcal{O}}$
\begin{equation}\label{special}
\begin{aligned}
f(\tau_\mathcal{O})^{\psi_N([R^{\widetilde{\nu}_N}]_N)}&=
f(\tau_\mathcal{O})^{\sigma_N([R^{\widetilde{\nu}_N}]_{\Gamma_1(N)})}
\quad\textrm{by Propositions \ref{isomorphism} and \ref{isomorphism2}}\\
&=f(\tau_\mathcal{O})^{\sigma_N([Q_1^{(\widetilde{\beta}_1)_N}]_{\Gamma_1(N)}
[Q_2^{(\widetilde{\beta}_2)_N}]_{\Gamma_1(N)})}\quad\textrm{by 
Lemma \ref{compatible}}\\
&=\left(f(\tau_\mathcal{O})^{\sigma_N([Q_1^{(\widetilde{\beta}_1)_N}]_{\Gamma_1(N)})}
\right)^{\sigma_N([Q_2^{(\widetilde{\beta}_2)_N}]_{\Gamma_1(N)})}\\
&\hspace{6cm}\textrm{by the homomorphism property of $\sigma_N$}\\
&=\left(f(\tau_\mathcal{O})^{\psi_N([Q_1^{(\widetilde{\beta}_1)_N}]_N)}
\right)^{\psi_N([Q_2^{(\widetilde{\beta}_2)_N}]_N)}
\quad\textrm{again by Propositions \ref{isomorphism} and \ref{isomorphism2}}\\
&=f(\tau_\mathcal{O})^{\psi_N(\rho_N(\mathcal{C})
\rho_N(\mathcal{C}'))}\quad\textrm{by \eqref{rQNrQN} and the homomorphism property of $\psi_N$}.
\end{aligned}
\end{equation}
Since $\widetilde{\nu}_N\equiv\nu\Mod{NM_2(\widehat{\mathbb{Z}})}$, 
we attain
\begin{equation*}
R^{\widetilde{\nu}_N}=
R\left(\mu\nu\begin{bmatrix}x\\y\end{bmatrix}\right)
~\textrm{with}~\mu=\widetilde{\nu}_N\nu^{-1}\in\mathrm{SL}_2(\widehat{\mathbb{Z}})~
\textrm{satisfying}~\mu\equiv I_2\Mod{NM_2(\widehat{\mathbb{Z}})}. 
\end{equation*}
Thus we find that
\begin{equation*}
\frac{\widetilde{b}-b_0}{2}\equiv
\frac{\widehat{b}-b_0}{2}\Mod{N\widehat{\mathbb{Z}}}, 
\end{equation*}
where $\widehat{b}$ is defined in \eqref{bb2a2}
(cf. \cite[Lemma 12.5]{J-K-S-Y}). It follows from 
\eqref{Qbhat}, \eqref{bb2a2}, \eqref{RCRCR}
together with  
$(\widetilde{\beta_i})_N
\equiv\beta_i
\Mod{NM_2(\widehat{\mathbb Z})}$ ($i=1,\,2$)
that 
\begin{equation}\label{2N2}
\frac{\widetilde{b}-b_0}{2}\equiv
\widehat{a}_2^{-1}\left(\frac{\widehat{b}_1-b_0}{2}\right)
+\left(\frac{\widehat{b}_2-b_0}{2}\right)
\equiv
a'^{-1}\left(\frac{b-b_0}{2}\right)
+\left(\frac{b'-b_0}{2}\right)
\Mod{N\widehat{\mathbb{Z}}}, 
\end{equation}
where $a'^{-1}$ is the inverse of $a'$ in $(\mathbb{Z}/N\mathbb{Z})^\times$. 
We further deduce that
\begin{equation}\label{special2}
\begin{aligned}
\left(\sqrt[N]{\mathfrak{t}}\right)^{\psi_N([R^{\widetilde{\nu}_N}]_N)}&=
\zeta_N^{\frac{\widetilde{b}-b_0}{2}}\sqrt[N]{\mathfrak{t}}\quad\textrm{by 
\eqref{RCRCR} and Proposition \ref{isomorphism2}}\\
&=\zeta_N^{a'^{-1}\left(\frac{b-b_0}{2}\right)
+\left(\frac{b'-b_0}{2}\right)}\sqrt[N]{\mathfrak{t}}\quad\textrm{by \eqref{2N2}}\\
&=\left(\sqrt[N]{\mathfrak{t}}\right)^{\psi_N(\rho_N(\mathcal{C})\rho_N(\mathcal{C}'))}
\quad\textrm{by \eqref{rQNrQN} and Definition \ref{Noperation2}}. 
\end{aligned}
\end{equation}
Therefore, we conclude by \eqref{special},
\eqref{special2}, Propositions \ref{generation} and  \ref{isomorphism2} that
\begin{equation*}
\rho_N(\mathcal{C})\rho_N(\mathcal{C}')=[R^{\widetilde{\nu}_N}]_N.
\end{equation*}
\end{proof}

\begin{theorem}\label{main}
Let $D$ be a negative integer such that $D\equiv1$ or $0\Mod{4}$, 
and let $K$ be an imaginary quadratic field given by $K=\mathbb{Q}(\sqrt{D})$. 
Let $\mathfrak{t}$ be a positive transcendental real number.
Then the binary operation on $\widehat{C}(D)$ defined in Definition \ref{explicit}
is well defined and makes
$\widehat{C}(D)$ a group isomorphic to $\mathrm{Gal}\left(K^\mathrm{ab}(\mathfrak{t}^{1/\infty})/K(\mathfrak{t})\right)$.
\end{theorem}
\begin{proof}
By Proposition \ref{bijective}, the map
\begin{equation*}
\rho:\widehat{C}(D)\rightarrow
\varprojlim_N C_{\Gamma(N)}(D,\,N),\quad
\mathcal{C}\mapsto(\rho_N(\mathcal{C}))_N
\end{equation*}
is bijective, where the inverse limit is endowed with the
coordinatewise group operation.
\par
Let $\mathcal{C},\,\mathcal{C}'\in\widehat{C}(D)$, and let
$(R,\,\nu)\in\widehat{\mathcal{Q}}(D)$ be obtained from any choice made in
the construction of Definition \ref{explicit}. 
By the definition of the map $\rho_N:\widehat{C}(D)\rightarrow C_{\Gamma(N)}(D,\,N)$
and Lemma \ref{mainlemma}, for every \(N\in\mathbb{N}\) we have
\begin{equation*}
\rho_N([(R,\,\nu)])=[R^{\widetilde{\nu}_N}]_N
=\rho_N(\mathcal{C})\rho_N(\mathcal{C}').
\end{equation*}
Therefore we obtain
\begin{equation}\label{rRnrCrc}
\rho([(R,\,\nu)])
=\rho(\mathcal{C})\rho(\mathcal{C}').
\end{equation}
The right-hand side depends only on $\mathcal{C}$ and $\mathcal{C}'$.
Since $\rho$ is injective, the class $[(R,\,\nu)]$ is independent of all
choices in Definition \ref{explicit}. Hence the binary operation on 
$\widehat{C}(D)$ is well
defined. Moreover, \eqref{rRnrCrc} gives
\begin{equation*}
\rho(\mathcal{C}\cdot\mathcal{C}')
=\rho(\mathcal{C})\rho(\mathcal{C}'),
\end{equation*}
which asserts that the bijection $\rho$ is a homomorphism. Consequently,
$\widehat{C}(D)$ equipped with this binary operation is in fact a group and
$\rho$ is an isomorphism of groups. 
\par
For each $N\in\mathbb{N}$, 
let $\psi_N:C_{\Gamma(N)}(D,\,N)\stackrel{\sim}{\rightarrow}
\mathrm{Gal}\left(K_{\mathcal{O},\,N}(\sqrt[N]{\mathfrak{t}})/K(\mathfrak{t})\right)$
be the isomorphism stated in Proposition \ref{isomorphism2}. 
The commutative diagram in Figure \ref{diagram1} ($\S$\ref{sect:inverse}) shows that these
isomorphisms are compatible with the transition maps. And, we achieve an 
isomorphism
\begin{equation*}
\varprojlim_N\psi_N:
\varprojlim_N C_{\Gamma(N)}(D,\,N)
\stackrel{\sim}{\rightarrow}
\varprojlim_N
\mathrm{Gal}\left(K_{\mathcal{O},\,N}(\sqrt[N]{\mathfrak{t}})/K(\mathfrak{t})\right).
\end{equation*}
By Lemma \ref{basic} (ii) and the definition of $\mathfrak{t}^{1/\infty}$, we see that
\begin{equation*}
\bigcup_{N\in\mathbb{N}}
K_{\mathcal{O},\,N}(\sqrt[N]{\mathfrak{t}})
=K^{\mathrm{ab}}(\mathfrak{t}^{1/\infty}).
\end{equation*}
Thus restriction induces an isomorphism
\begin{equation*}
\Phi:\mathrm{Gal}\left(K^{\mathrm{ab}}(\mathfrak{t}^{1/\infty})/K(\mathfrak{t})\right)
\stackrel{\sim}{\rightarrow}
\varprojlim_N
\mathrm{Gal}\left(K_{\mathcal{O},\,N}(\sqrt[N]{\mathfrak{t}})/K(\mathfrak{t})\right).
\end{equation*}
Therefore the composition map
\begin{equation}\label{constructeta}
\eta=\Phi^{-1}\circ\left(\varprojlim_N\psi_N\right)\circ\rho:\widehat{C}(D)
\rightarrow\mathrm{Gal}\left(K^{\mathrm{ab}}(\mathfrak{t}^{1/\infty})/K(\mathfrak{t})\right)
\end{equation}
is an isomorphism from $\widehat{C}(D)$ onto
$\mathrm{Gal}\left(K^\mathrm{ab}(\mathfrak{t}^{1/\infty})/K(\mathfrak{t})\right)$,
as required.
\end{proof}

\begin{definition}
We call the group $\widehat{C}(D)$ given in Theorem \ref{main}
the \textit{adelic framed form class group of discriminant $D$}. 
\end{definition}

\begin{remark}
The identity element of $\widehat{C}(D)$ is $[(Q_0,\,I_2)]$ since
\begin{equation*}
\rho([(Q_0,\,I_2)])=([Q_0]_N)_N,
\end{equation*} 
which is the identity element of $\varprojlim_N C_{\Gamma(N)}(D,\,N)$. 
\end{remark}

\section {The Shimura reciprocity law}\label{sect:Shimura}

We give a form-theoretic formulation of the Shimura reciprocity law
in terms of adelic framed form classes by identifying a subgroup of $\widehat{C}(D)$ with 
$\mathrm{Gal}(K^\mathrm{ab}/K)$ and 
explicitly describing the corresponding Galois action on special values of modular functions.
\par
For $(Q,\,\gamma)\in\widehat{\mathcal{Q}}(D)$, we write
\begin{equation*}
Q\left(\gamma\begin{bmatrix}x\\y\end{bmatrix}\right)=
a_{(Q,\,\gamma)}x^2+b_{(Q,\,\gamma)}xy+c_{(Q,\,\gamma)}y^2\quad(\in\widehat{\mathbb{Z}}[x,\,y]). 
\end{equation*}
Observe that $a_{(Q,\,\gamma)}$, $b_{(Q,\,\gamma)}$, $c_{(Q,\,\gamma)}$ depend
only on the class $[(Q,\,\gamma)]$ in $\widehat{C}(D)$. 
Furthermore, we define an element $M_{(Q,\,\gamma)}$
of $\mathrm{GL}_2(\widehat{\mathbb{Z}})$ by 
\begin{equation}\label{MQ}
M_{(Q,\,\gamma)}=\begin{bmatrix}
1&-a_{(Q,\,\gamma)}^{-1}\left(\frac{b_{(Q,\,\gamma)}+b_0}{2}\right)\\0&a_{(Q,\,\gamma)}^{-1} 
\end{bmatrix}E\gamma^{-1}E,
\end{equation}
where $E=\begin{bmatrix}1&0\\0&-1\end{bmatrix}$. 
\par
Put
\begin{equation*}
\mathcal{F}=\bigcup_{N\in\mathbb{N}}\mathcal{F}_N
\quad\textrm{and}\quad
\mathcal{F}_{\tau_\mathcal{O}}=\bigcup_{N\in\mathbb{N}}\mathcal{F}_{N,\,
\tau_\mathcal{O}}. 
\end{equation*}
Since $\mathcal{F}_N\subseteq\mathcal{F}_M$ for all $M,\,N
\in\mathbb{N}$ with $N\,|\,M$, 
the field $\mathcal{F}$ is Galois over $\mathcal{F}_1$. 
By \eqref{GFF} and Proposition \ref{GFaction}, 
there is an isomorphism
\begin{align*}
\mathrm{GL}_2(\widehat{\mathbb{Z}})/\langle-I_2\rangle
&\stackrel{\sim}{\rightarrow}\mathrm{Gal}(\mathcal{F}/\mathcal{F}_1)\\
[\gamma]&\mapsto\bigg(
f\mapsto f^{\gamma_N}~|~f\in\mathcal{F}_N~(N\in\mathbb{N})\bigg)
\quad(\gamma\in\mathrm{GL}_2(\widehat{\mathbb{Z}})),
\end{align*}
where $\gamma_N$ is the image of $\gamma$ in 
$\mathrm{GL}_2(\mathbb{Z}/N\mathbb{Z})$
(cf. \cite[Proposition 6.21 and Theorem 6.23]{Shimura}).
Then the Shimura reciprocity law \cite[Theorem 6.31]{Shimura} takes
the following concrete form. 

\begin{theorem}\label{Shimurarec}
The subset
\begin{equation*}
H=\left\{[(Q,\,\gamma)]~|~(Q,\,\gamma)\in\widehat{\mathcal{Q}}(D),~
b_{(Q,\,\gamma)}=b_0\right\}
\end{equation*}
of $\widehat{C}(D)$ is a subgroup of $\widehat{C}(D)$ and is isomorphic 
to $\mathrm{Gal}(K^{\mathrm{ab}}/K)$ via  
\begin{align*}
H&\rightarrow\mathrm{Gal}(K^\mathrm{ab}/K)\\
[(Q,\,\gamma)]&\mapsto\bigg(f(\tau_\mathcal{O})\mapsto
f^{[M_{(Q,\,\gamma)}]}(-\overline{\omega}_Q)~|~f\in\mathcal{F}_{\tau_\mathcal{O}}\bigg).
\end{align*}
\end{theorem}
\begin{proof}
Recall from the definition \eqref{constructeta} that the isomorphism
\begin{equation*}
\eta:\widehat{C}(D)\stackrel{\sim}{\rightarrow}
\mathrm{Gal}\left(K^\mathrm{ab}(\mathfrak{t}^{1/\infty})/K(\mathfrak{t})\right)
\end{equation*}
is the unique isomorphism for which the following diagram commutes:
\begin{figure}[H]
\begin{equation*}
\xymatrixcolsep{5pc}
\xymatrix{
\widehat{C}(D)\ar@{->}[rr]^{\sim}_{\left(\rho_N\right)_N}
\ar@{->}[dd]_{}^{\eta} && \varprojlim_{N\in\mathbb{N}} C_{\Gamma(N)}(D,\,N) \ar@{->}[dd]^{\rotatebox{90}{$\sim$}}_{\varprojlim_N \psi_N} \\\\
\mathrm{Gal}\left(K^\mathrm{ab}(\mathfrak{t}^{1/\infty})/K(\mathfrak{t})\right)
 \ar@{->}[rr]^{\sim}_{\left(\cdot|_{K_{\mathcal{O},\,N}(\sqrt[N]{\mathfrak{t}})}\right)_N} && \varprojlim_{N\in\mathbb{N}}\mathrm{Gal}\left(K_{\mathcal{O},\,N}(\sqrt[N]{\mathfrak{t}})/K(\mathfrak{t})\right)
}
\end{equation*}
\caption{A commutative diagram identifying $\widehat{C}(D)$
with a Galois group}
\label{diagram2}
\end{figure}
Let $(Q,\,\gamma)\in\widehat{\mathcal{Q}}(D)$. For each $N\in\mathbb{N}$, take a matrix $\gamma^*_N\in\mathrm{SL}_2(\mathbb{Z})$
satisfying 
\begin{equation}\label{2N}
\gamma^*_N\equiv\gamma\Mod{2NM_2(\widehat{\mathbb{Z}})}.
\end{equation} 
Since $\gamma^*_N\equiv\gamma\Mod{NM_2(\widehat{\mathbb{Z}})}$, we get
\begin{equation}\label{Qr*}
\rho_N([(Q,\,\gamma)])=
[Q^{\gamma^*_N}]_N. 
\end{equation}
We then deduce that
\begin{eqnarray*}
&&[(Q,\,\gamma)]\in\eta^{-1}\left(\mathrm{Gal}\left(K^\mathrm{ab}(\mathfrak{t}^{1/\infty})/K(\mathfrak{t}^{1/\infty})\right)\right)\\
&\Longleftrightarrow&\eta([(Q,\,\gamma)])\in\mathrm{Gal}\left(K^\mathrm{ab}(\mathfrak{t}^{1/\infty})/K(\mathfrak{t}^{1/\infty})\right)\\
&\Longleftrightarrow&\left(\sqrt[N]{\mathfrak{t}}\right)^{\psi_N(\rho_N([(Q,\,\gamma)]))}=\sqrt[N]{\mathfrak{t}}
~~\textrm{for all}~N\in\mathbb{N}~
\textrm{by the commutative diagram in Figure \ref{diagram2}}\\
&\Longleftrightarrow&\zeta_N^{\frac{b_{Q^{\gamma^*_N}}-b_0}{2}}\sqrt[N]{\mathfrak{t}}=\sqrt[N]{\mathfrak{t}}~~
\textrm{for all $N\in\mathbb{N}$ by Proposition \ref{isomorphism2} and \eqref{Qr*}},\\
&&\hspace{3.4cm}\textrm{where $b_{Q^{\gamma^*_N}}$
is the coefficient of $xy$ in $Q^{\gamma^*_N}$}\\
&\Longleftrightarrow&
b_{Q^{\gamma^*_N}}\equiv b_0\Mod{2N}~\textrm{for all}~N\in\mathbb{N}\\
&\Longleftrightarrow&b_{(Q,\,\gamma)}\equiv b_0\Mod{2N\widehat{\mathbb{Z}}}~\textrm{by \eqref{2N}}\\
&\Longleftrightarrow&b_{(Q,\,\gamma)}=b_0\quad
\textrm{because}~\bigcap_{N\in\mathbb{N}}2N\widehat{\mathbb{Z}}=\{0\}\\
&\Longleftrightarrow&[(Q,\,\gamma)]\in H. 
\end{eqnarray*}
This yields that
$H$ is a subgroup of $\widehat{C}(D)$ isomorphic to 
$\mathrm{Gal}\left(K^\mathrm{ab}(\mathfrak{t}^{1/\infty})/K(\mathfrak{t}^{1/\infty})\right)$. 
\par
Now, let $[(Q,\,\gamma)]\in H$. 
By Lemma \ref{basic} (ii), Proposition \ref{generation} and the diagram in
Figure \ref{diagram2}, $\eta([(Q,\,\gamma)])$ is completely determined by
its action on the values $f(\tau_{\mathcal{O}})$ for
$f\in\mathcal F_{\tau_{\mathcal O}}$.
Let $f\in\mathcal{F}_{N,\,\tau_\mathcal{O}}$ for some $N\in\mathbb{N}$. 
We find that
\begin{align*}
f(\tau_\mathcal{O})^{\eta([(Q,\,\gamma)])}&=
f(\tau_\mathcal{O})^{\psi_N(\rho_N([(Q,\,\gamma)]))}
\quad\textrm{by the commutative diagram in Figure \ref{diagram2}}\\
&=f(\tau_\mathcal{O})^{\psi_N([Q^{\widetilde{\gamma}_N}]_N)}\\
&=f^{m_{Q^{\widetilde{\gamma}_N}}}(-\overline{\omega}_{Q^{\widetilde{\gamma}_N}})
\quad\textrm{by Proposition \ref{isomorphism2}}\\
&=f^{m_{Q^{\widetilde{\gamma}_N}}}(-\overline{\widetilde{\gamma}_N^{-1}(\omega_Q)})\\
&=f^{m_{Q^{\widetilde{\gamma}_N}}}((E\widetilde{\gamma}_N^{-1}E)(-\overline{\omega}_Q))\\
&=f^{m_{Q^{\widetilde{\gamma}_N}}(E\widetilde{\gamma}_N^{-1}E)}(-\overline{\omega}_Q)
\quad\textrm{by Proposition \ref{GFaction} (ii)}\\
&=f^{[M_{(Q,\,\gamma)}]}(-\overline{\omega}_Q)
\quad\textrm{by the definition \eqref{MQ}}. 
\end{align*}
Since $\mathfrak{t}$ is a transcendental real number, we have
$K^\mathrm{ab}\cap K(\mathfrak{t}^{1/\infty})=K$. Hence it follows that
\begin{equation}\label{GKKGKK}
\mathrm{Gal}\left(K^\mathrm{ab}(\mathfrak{t}^{1/\infty})/K(\mathfrak{t}^{1/\infty})\right)
\simeq\mathrm{Gal}(K^\mathrm{ab}/K)\quad\textrm{via restriction}
\end{equation}
(see \cite[Theorem 1.12 in Chapter VI]{Lang}). 
This completes the proof. 
\end{proof}

\section {Topology on $\widehat{C}(D)$}

With the Krull topology, $\mathrm{Gal}\left(K^\mathrm{ab}(\mathfrak{t}^{1/\infty})/K(\mathfrak{t})\right)$
is canonically a topological group. 
In this section, we endow $\widehat{C}(D)$ with a topology
so that the group isomorphism
\begin{equation*}
\widehat{C}(D)
\stackrel{\sim}{\rightarrow}\mathrm{Gal}\left(K^\mathrm{ab}(\mathfrak{t}^{1/\infty})/K(\mathfrak{t})\right)
\end{equation*}
established in Theorem \ref{main}
becomes an isomorphism of topological groups. 

\begin{definition}\label{topology}
We define a topology on $\widehat{C}(D)$ as follows.
First, equip $\mathcal{Q}(D)$ and $\mathrm{SL}_2(\widehat{\mathbb{Z}})$
with the discrete topology and the profinite topology
described in Remark \ref{alternative}, respectively.
Then we regard $\widehat{\mathcal{Q}}(D)$
as a subspace of the product space $\mathcal{Q}(D)
\times \mathrm{SL}_2(\widehat{\mathbb{Z}})$.
Finally, give $\widehat{C}(D)=\widehat{\mathcal{Q}}(D)/\sim$
the quotient topology
induced by the canonical surjection $\widehat{\mathcal{Q}}(D)
\rightarrow\widehat{\mathcal{Q}}(D)/\sim$.
\end{definition}

We consider
$\varprojlim_N\mathrm{Gal}\left(K_{\mathcal{O},\,N}(\sqrt[N]{\mathfrak{t}})/K(\mathfrak{t})\right)$
as a profinite group
by giving $\mathrm{Gal}\left(K_{\mathcal{O},\,N}(\sqrt[N]{\mathfrak{t}})/
K(\mathfrak{t})\right)$ the discrete topology for every $N\in\mathbb{N}$. 
Since the family 
$\left\{K_{\mathcal{O},\,N}(\sqrt[N]{\mathfrak{t}})\right\}_N$
is cofinal in the family of finite Galois subextensions of 
$K^\mathrm{ab}(\mathfrak{t}^{1/\infty})/K(\mathfrak{t})$,
the restriction maps induce an isomorphism
\begin{equation*}
\mathrm{Gal}\left(K^\mathrm{ab}(\mathfrak{t}^{1/\infty})/K(\mathfrak{t})\right)
\stackrel{\sim}{\rightarrow}
\varprojlim_N\mathrm{Gal}\left(K_{\mathcal{O},\,N}(\sqrt[N]{\mathfrak{t}})/K(\mathfrak{t})\right)
\end{equation*}
of topological groups;
see \cite[Chapter 7]{Milne2} and \cite[Lemma 1.1.9 and Theorem 2.11.1]{R-Z}. 

\begin{theorem}\label{topological}
Let $D$ be a negative integer such that $D\equiv1$ or $0\Mod{4}$, 
and let $K$ be an imaginary quadratic field given by $K=\mathbb{Q}(\sqrt{D})$. 
Let $\mathfrak{t}$ be a positive transcendental real number.
Then $\widehat{C}(D)$ endowed with
the topology defined in Definition \ref{topology} 
is isomorphic to 
$\mathrm{Gal}\left(K^\mathrm{ab}(\mathfrak{t}^{1/\infty})
/K(\mathfrak{t})\right)$ as a topological group. 
\end{theorem}
\begin{proof}
First, we show that $\widehat{C}(D)$ is compact.
Choose representatives $Q_1,\,Q_2,\,\ldots,\,Q_h\in\mathcal{Q}(D)$ of the
elements of the finite group $C_{\Gamma(1)}(D,\,1)$, and put
\begin{equation*}
X_i=\left\{\beta\in\mathrm{SL}_2(\widehat{\mathbb{Z}})~|~
(Q_i,\,\beta)\in\widehat{\mathcal{Q}}(D)\right\}
=\left\{\beta\in\mathrm{SL}_2(\widehat{\mathbb{Z}})~|~
Q_i\left(\beta\begin{bmatrix}1\\0\end{bmatrix}\right)\in\widehat{\mathbb{Z}}^\times
\right\}\quad(i=1,\,2,\,\ldots,\,h). 
\end{equation*}
For each $(Q,\,\gamma)\in\widehat{\mathcal{Q}}(D)$, since
\begin{equation*}
Q=Q_i^\alpha\quad\textrm{for some}~i\in\{1,\,2,\,\ldots,\,h\}~
\textrm{and}~\alpha\in\mathrm{SL}_2(\mathbb{Z}), 
\end{equation*}
we have
\begin{equation*}
[(Q,\,\gamma)]=[(Q,\,\gamma)^{\alpha^{-1}}]=[(Q_i,\,\alpha\gamma)],
\end{equation*}
which implies $\alpha\gamma\in X_i$. 
Thus, if we let 
\begin{equation*}
q:\widehat{\mathcal{Q}}(D)\rightarrow\widehat{C}(D)
\end{equation*}
be the canonical quotient map, which is continuous, then we obtain
\begin{equation}\label{union}
\widehat{C}(D)=\bigcup_{i=1}^h
q(\{Q_i\}\times X_i).
\end{equation}
Now, since the polynomial map 
\begin{equation*}
g_i:\mathrm{SL}_2(\widehat{\mathbb{Z}})
\rightarrow\widehat{\mathbb{Z}},\quad
\gamma\mapsto Q_i\left(\gamma\begin{bmatrix}1\\0\end{bmatrix}\right)
\end{equation*}
is continuous and 
$\widehat{\mathbb{Z}}^\times$ is closed in 
$\widehat{\mathbb{Z}}$ (because
$\widehat{\mathbb{Z}}^\times$ is compact in the Hausdorff space 
$\widehat{\mathbb{Z}}$), the inverse image $X_i=g_i^{-1}(\widehat{\mathbb{Z}}^\times)$
is closed in $\mathrm{SL}_2(\widehat{\mathbb{Z}})$ for every $i=1,\,2,\,\ldots,\,h$.
So, $X_i$ is compact because it is a closed subset of
the compact space $\mathrm{SL}_2(\widehat{\mathbb{Z}})$.
And, it follows that the image of the compact subspace $\{Q_i\}\times X_i
\simeq X_i$
of $\widehat{\mathcal{Q}}(D)$
under the continuous map $q$ is compact. 
Hence $\widehat{C}(D)$ is compact by \eqref{union}.
\par
Next, we assert that for each $N\in\mathbb{N}$
\begin{equation*}
\rho_N:\widehat{C}(D)\rightarrow C_{\Gamma(N)}(D,\,N)
\end{equation*}
is continuous, where $C_{\Gamma(N)}(D,\,N)$ is given the discrete
topology. 
Indeed, put
\begin{equation*}
T_N=\left\{(Q,\,\alpha)\in\mathcal{Q}(D)\times
\mathrm{SL}_2(\mathbb{Z}/N\mathbb{Z})~|~
Q^{\widetilde{\alpha}}\in\mathcal{Q}(D,\,N)\right\},
\end{equation*}
where $\widetilde{\alpha}\in\mathrm{SL}_2(\mathbb{Z})$ is any lift of $\alpha$. 
Since $T_N$ is a subspace of the discrete space $\mathcal{Q}(D)\times
\mathrm{SL}_2(\mathbb{Z}/N\mathbb{Z})$, it is discrete too. 
Note that the composition map
\begin{equation*}
\rho_N\circ q:\widehat{\mathcal{Q}}(D)
\rightarrow C_{\Gamma(N)}(D,\,N),
\quad (Q,\,\gamma)\mapsto[Q^{\widetilde{\gamma}_N}]_N
\end{equation*}
factors through the well-defined continuous reduction map
\begin{equation*}
f_N:\widehat{\mathcal{Q}}(D)\rightarrow
T_N,\quad
(Q,\,\gamma)\mapsto (Q,\,\gamma_N),
\end{equation*}
where $\gamma_N$ is the image of $\gamma$ in $\mathrm{SL}_2(\mathbb{Z}/N\mathbb{Z})$. 
That is, if we define
\begin{equation*}
h_N:T_N\rightarrow C_{\Gamma(N)}(D,\,N),
\quad
(Q,\,\alpha)\mapsto[Q^{\widetilde{\alpha}}]_N,
\end{equation*}
then the following diagram commutes. 
\begin{figure}[H]
\begin{equation*}
\xymatrixcolsep{5pc}
\xymatrix{
\widehat{\mathcal{Q}}(D)\ar@{->}[rr]^{\rho_N\circ q}
\ar@{->}[ddr]_{f_N} && C_{\Gamma(N)}(D,\,N)\\\\
&T_N \ar@{->}[uur]_{h_N} &
}
\end{equation*}
\caption{A commutative diagram for continuity of $\rho_N$}
\label{diagram3}
\end{figure}
Since both $T_N$ and $C_{\Gamma(N)}(D,\,N)$ are discrete,
the map $h_N$ given in the commutative diagram in Figure \ref{diagram3}
is continuous. So it follows that
$\rho_N\circ q=h_N\circ f_N$ is also continuous. Hence $\rho_N$ is continuous
because $q$ is a quotient map.
\par
Hence, we are ready to show that the group isomorphism
\begin{equation*}
\eta:\widehat{C}(D)\stackrel{\sim}{\rightarrow}
\mathrm{Gal}\left(K^\mathrm{ab}(\mathfrak{t}^{1/\infty})/K(\mathfrak{t})\right)
\end{equation*}
defined in \eqref{constructeta} is continuous.  
Let
\begin{equation*}
\Psi:\widehat{C}(D)\stackrel{\sim}{\rightarrow}\varprojlim_N
\mathrm{Gal}\left(K_{\mathcal{O},\,N}(\sqrt[N]{\mathfrak{t}})/K(\mathfrak{t})\right),
\quad
\mathcal{C}\mapsto
\left(\psi_N(\rho_N(\mathcal{C}))\right)_N
\end{equation*}
be the isomorphism obtained from the diagram in Figure \ref{diagram2} ($\S$\ref{sect:Shimura}). 
For each $N\in\mathbb{N}$, since $\rho_N$ is continuous and $\psi_N$ is an isomorphism between
finite discrete spaces,
the composition map
\begin{equation*}
\psi_N\circ\rho_N:\widehat{C}(D)
\rightarrow\mathrm{Gal}\left(K_{\mathcal{O},\,N}(\sqrt[N]{\mathfrak{t}})/K(\mathfrak{t})\right)
\end{equation*}
is also  continuous.
Thus the map $\Psi$ is
continuous because 
the inverse limit has the subspace topology inherited from $\prod_N\mathrm{Gal}\left(K_{\mathcal{O},\,N}(\sqrt[N]{\mathfrak{t}})/K(\mathfrak{t})\right)$.
Let
\begin{equation*}
\Phi:\mathrm{Gal}\left(K^\mathrm{ab}(\mathfrak{t}^{1/\infty})/K(\mathfrak{t})\right)
\stackrel{\sim}{\rightarrow}
\varprojlim_N\mathrm{Gal}\left(K_{\mathcal{O},\,N}(\sqrt[N]{\mathfrak{t}})/K(\mathfrak{t})\right)
\end{equation*}
be the homeomorphism induced by restriction.
Since $\Phi\circ\eta=\Psi$ by the commutativity of the diagram in Figure \ref{diagram2},
$\eta$ is continuous.
\par
Observe that $\mathrm{Gal}\left(K^\mathrm{ab}(\mathfrak{t}^{1/\infty})/K(\mathfrak{t})\right)$
is Hausdorff, since it is a profinite group. Eventually the map
$\eta$ is a continuous bijection from the compact space
$\widehat{C}(D)$ onto the Hausdorff space 
$\mathrm{Gal}\left(K^\mathrm{ab}(\mathfrak{t}^{1/\infty})/K(\mathfrak{t})\right)$. 
Therefore $\eta$ is a closed map, and so
it is a homeomorphism.
Finally, we conclude that 
$\widehat{C}(D)$ is isomorphic to $\mathrm{Gal}\left(K^\mathrm{ab}(\mathfrak{t}^{1/\infty})/K(\mathfrak{t})\right)$ as a topological group.
\end{proof}

\section{Arithmetic rigidity of $\widehat C(D)$}

The Kummer direction $\mathfrak{t}^{1/\infty}$ records the cyclotomic
action of $\operatorname{Gal}(K^\mathrm{ab}/K)$. In this final section, we
prove that the group $\widehat{C}(D)$ uniquely determines the imaginary
quadratic field $K$.
\par
For simplicity, let
\begin{align*}
\mathcal{G}_K&=\mathrm{Gal}\left(
K^\mathrm{ab}(\mathfrak{t}^{1/\infty})/K(\mathfrak{t})\right),\\
\mathcal{A}_K&=\mathrm{Gal}\left(
K^\mathrm{ab}(\mathfrak{t}^{1/\infty})/K(\mathfrak{t}^{1/\infty})\right)
\quad(\leq \mathcal{G}_K),\\
\mathcal{K}_K&=\mathrm{Gal}\left(
K^\mathrm{ab}(\mathfrak{t}^{1/\infty})/K^\mathrm{ab}(\mathfrak{t})\right)
\quad(\leq \mathcal{G}_K).
\end{align*}
Recall from \eqref{GKKGKK} that
\begin{equation}\label{AKG}
\mathcal{A}_K\simeq
\mathrm{Gal}(K^\mathrm{ab}/K)\quad\textrm{via restriction}.
\end{equation}
Now that $K/\mathbb{Q}$ is abelian, we get
$K\subseteq\mathbb{Q}^\mathrm{ab}\subseteq K^\mathrm{ab}$.
Let
\begin{equation*}
\iota:\mathrm{Gal}(\mathbb{Q}^\mathrm{ab}/\mathbb{Q})
\stackrel{\sim}{\rightarrow}\widehat{\mathbb{Z}}^\times
\end{equation*}
be the canonical cyclotomic isomorphism, and define a homomorphism
\begin{equation*}
\chi_K:\mathrm{Gal}(K^\mathrm{ab}/K)
\rightarrow\widehat{\mathbb{Z}}^\times
\end{equation*}
as the composition of the restriction $\mathrm{Gal}(K^\mathrm{ab}/K)\rightarrow
\mathrm{Gal}(\mathbb{Q}^\mathrm{ab}/K)$
with $\iota$. Put
\begin{equation}\label{IKdef}
I_K=\mathrm{im}(\chi_K).
\end{equation}
Since
\begin{equation}\label{IK}
I_K=\iota\left(\mathrm{Gal}(\mathbb{Q}^\mathrm{ab}/K)\right)
=\iota\left(
\ker\left(
\mathrm{Gal}(\mathbb{Q}^\mathrm{ab}/\mathbb{Q})
\stackrel{\mathrm{restriction}}{\longrightarrow}
\mathrm{Gal}(K/\mathbb{Q})
\right)\right),
\end{equation}
$I_K$ is a subgroup of index $2$ in $\widehat{\mathbb{Z}}^\times$. Here the map $\chi_K$ is the profinite
cyclotomic character of $K$; see \cite[$\S$V.5]{Milne}.
\par
For each $u\in\widehat{\mathbb{Z}}$, let $\eta_u$ denote the unique
 element of $\mathcal{K}_K$ satisfying
\begin{equation*}
\left(\sqrt[N]{\mathfrak{t}}\right)^{\eta_u}
=\zeta_N^{u_N}\sqrt[N]{\mathfrak{t}}
\quad(N\in\mathbb{N}),
\end{equation*}
where $u_N$ is the image of $u$ in $\mathbb{Z}/N\mathbb{Z}$. Since
\begin{equation*}
\mathrm{Gal}\left(
K^\mathrm{ab}(\sqrt[N]{\mathfrak{t}})/K^\mathrm{ab}(\mathfrak{t})
\right)\simeq\mathbb{Z}/N\mathbb{Z}
\quad(N\in\mathbb{N})
\end{equation*}
and the restriction maps correspond to the natural reduction maps, we
obtain an isomorphism 
\begin{equation}\label{ZKue}
\widehat{\mathbb{Z}}\stackrel{\sim}{\rightarrow}\mathcal{K}_K,
\quad u\mapsto\eta_u.
\end{equation}
Using this isomorphism, we define a homomorphism
\begin{equation*}
\alpha_K:\mathcal{A}_K\rightarrow\mathrm{Aut}(\mathcal{K}_K)
\end{equation*}
by
\begin{equation*}
\alpha_K(\sigma)(\eta_u)=
\eta_{\chi_K(\sigma^{-1}|_{K^\mathrm{ab}})u}
\quad(u\in\widehat{\mathbb{Z}}).
\end{equation*}

\begin{lemma}\label{decomposition}
We have a semidirect product decomposition
\begin{equation*}
\mathcal{G}_K=\mathcal{K}_K\rtimes_{\alpha_K}\mathcal{A}_K.
\end{equation*}
\end{lemma}
\begin{proof}
It is straightforward that
$\mathcal{A}_K\cap\mathcal{K}_K=\{1\}$ and
\begin{equation*}
\mathcal{K}_K=
\ker\left(
\mathcal{G}_K\stackrel{\mathrm{restriction}}{\longrightarrow}
\mathrm{Gal}(K^\mathrm{ab}/K)
\right)\trianglelefteq \mathcal{G}_K.
\end{equation*}
Let $\nu\in \mathcal{G}_K$. By \eqref{AKG}, there is an element
$\sigma\in\mathcal{A}_K$ such that
$\nu|_{K^\mathrm{ab}}=\sigma|_{K^\mathrm{ab}}$. Hence
we see that
\begin{equation*}
\nu=(\nu\sigma^{-1})\sigma\in\mathcal{K}_K\mathcal{A}_K,
\end{equation*}
which shows that $\mathcal{G}_K=\mathcal{K}_K\mathcal{A}_K$.
\par
And, it remains to identify the conjugation action of $\mathcal{A}_K$ on $\mathcal{K}_K$. 
Let $\sigma\in\mathcal{A}_K$ and $\eta_u\in\mathcal{K}_K$ with $u\in\widehat{\mathbb{Z}}$. 
We find that for each $N\in\mathbb{N}$
\begin{align*}
\left(\sqrt[N]{\mathfrak{t}}\right)^{\sigma\eta_u\sigma^{-1}}
&=\left(\left(\sqrt[N]{\mathfrak{t}}\right)^{\eta_u}\right)^{\sigma^{-1}}
\quad\textrm{because $\sigma$ fixes $\sqrt[N]{\mathfrak{t}}$}\\
&=\left(\zeta_N^{u_N}\sqrt[N]{\mathfrak{t}}\right)^{\sigma^{-1}}\\
&=\zeta_N^{\left(\chi_K(\sigma^{-1}|_{K^\mathrm{ab}})\right)_Nu_N}
\sqrt[N]{\mathfrak{t}}\\
&=\left(\sqrt[N]{\mathfrak{t}}\right)^{
\eta_{\chi_K(\sigma^{-1}|_{K^\mathrm{ab}})u}}.
\end{align*}
Therefore we obtain
\begin{equation*}
\sigma\eta_u\sigma^{-1}
=\eta_{\chi_K(\sigma^{-1}|_{K^\mathrm{ab}})u}=\alpha_K(\sigma)(\eta_u).
\end{equation*}
\end{proof}

For $g,\,h\in \mathcal{G}_K$, let
$[g,h]=g^{-1}h^{-1}gh$ be the commutator. By $[\mathcal{G}_K,\,\mathcal{G}_K]$ 
we mean the commutator subgroup of $\mathcal{G}_K$, that is,
\begin{equation*}
[\mathcal{G}_K,\mathcal{G}_K]=\langle[g,\,h]~|~g,h\in \mathcal{G}_K\rangle
\quad(\trianglelefteq \mathcal{G}_K). 
\end{equation*}

\begin{lemma}\label{comm}
Let $J_K$ be the ideal of $\widehat{\mathbb{Z}}$ generated by 
$\{a-1~|~a\in I_K\}$, namely,
\begin{equation*}
J_K=(a-1~|~a\in I_K).
\end{equation*}
\begin{enumerate}
\item[\textup{(i)}] We establish $[\mathcal{G}_K,\,\mathcal{G}_K]=\{\eta_u~|~u\in J_K\}$.
\item[\textup{(ii)}] We get
$J_K=e\widehat{\mathbb Z}$ for some $e\in\mathbb{N}$,
and so $[\mathcal{G}_K,\,\mathcal{G}_K]\simeq\widehat{\mathbb{Z}}$.
\end{enumerate}
\end{lemma}
\begin{proof}
\begin{enumerate}
\item[(i)]
Let
\begin{equation*}
M=\{\eta_u~|~u\in J_K\}\subseteq\mathcal{K}_K.
\end{equation*}
We claim that $M$ is normal in $\mathcal{G}_K$. Indeed, conjugation by
an element of $\mathcal K_K$ acts trivially on $M$, since
$\mathcal K_K\simeq\widehat{\mathbb{Z}}$ is abelian. Moreover, for
$\sigma\in\mathcal{A}_K$ and $v\in J_K$, Lemma \ref{decomposition} gives
\begin{equation*}
\sigma\eta_v\sigma^{-1}=
\eta_{\chi_K(\sigma^{-1}|_{K^\mathrm{ab}})\textstyle v}\in M,
\end{equation*}
because $J_K$ is an ideal of $\widehat{\mathbb Z}$.
Since $\mathcal{G}_K=\mathcal{K}_K\mathcal{A}_K$, this yields that
$M$ is normal in $\mathcal{G}_K$. 
\par
Observe by Lemma \ref{decomposition} that
for any $\eta_u\in\mathcal{K}_K$ ($u\in\widehat{\mathbb{Z}}$) and $\sigma\in\mathcal{A}_K$
\begin{equation}\label{etasigma}
[\eta_u,\,\sigma]=\eta_u^{-1}\left(\sigma^{-1}\eta_u\sigma\right)
=\eta_{-u}\eta_{\chi_K(\sigma|_{K^\mathrm{ab}})u}
=\eta_{\left(\chi_K(\sigma|_{K^\mathrm{ab}})-1\right)u}\in M.
\end{equation}
Furthermore, since $\mathcal{K}_K$ and $\mathcal{A}_K$ are abelian, \eqref{etasigma}
asserts that
their images in $\mathcal{G}_K/M$ commute with each other.
Since these images generate $\mathcal{G}_K/M$, the quotient $\mathcal{G}_K/M$ is
abelian. Thus we obtain the inclusion
$[\mathcal{G}_K,\,\mathcal{G}_K]\subseteq M$.
\par
Conversely, every $v\in J_K$ is of the form
\begin{equation}\label{v}
v=\sum_{i=1}^r(a_i-1)u_i
\quad(r\in\mathbb{N},~a_i\in I_K,~u_i\in\widehat{\mathbb{Z}}).
\end{equation}
By \eqref{AKG} and the definition \eqref{IKdef}, for each $i=1,\,2,\,\ldots,\,r$ we may choose
$\sigma_i\in\mathcal{A}_K$ satisfying
\begin{equation}\label{chiai}
\chi_K(\sigma_i|_{K^\mathrm{ab}})=a_i.
\end{equation}
We then achieve that
\begin{align*}
\eta_v&=\eta_{\sum_{i=1}^r(\chi_K(\sigma_i|_{K^\mathrm{ab}})-1)u_i}\quad\textrm{by 
\eqref{v} and \eqref{chiai}}\\
&=\prod_{i=1}^r\eta_{(\chi_K(\sigma_i|_{K^\mathrm{ab}})-1)u_i}\quad\textrm{via the isomorphism given in \eqref{ZKue}}\\
&=\prod_{i=1}^r[\eta_{u_i},\,\sigma_i]\quad\textrm{by \eqref{etasigma}}\\
&\in [\mathcal{G}_K,\,\mathcal{G}_K].
\end{align*}
This shows that $M\subseteq[\mathcal{G}_K,\,\mathcal{G}_K]$, and hence
\begin{equation*}
[\mathcal{G}_K,\,\mathcal{G}_K]=M=\{\eta_u~|~u\in J_K\}. 
\end{equation*}
\item[(ii)]
Since $I_K$ has index $2$ in $\widehat{\mathbb{Z}}^{\times}$, it contains every square.
Identifying $\widehat{\mathbb{Z}}$ with 
$\prod_{p\,:\,\textrm{primes}}\mathbb{Z}_p$, 
let $b=(b_p)_p\in\widehat{\mathbb{Z}}^{\times}$ be defined by
\begin{equation*}
b_2=3\quad\textrm{and}\quad
b_p=2~(p\neq2).
\end{equation*}
Then we derive that
\begin{equation*}
J_K\supseteq(b^2-1)\widehat{\mathbb{Z}}=
8\mathbb{Z}_2\times3\mathbb{Z}_3
\times\prod_{p\neq2,\,3}\mathbb{Z}_p=
24\widehat{\mathbb{Z}}.
\end{equation*}
Therefore $J_K/24\widehat{\mathbb{Z}}$ is an ideal of $\widehat{\mathbb{Z}}/24\widehat{\mathbb{Z}}
\simeq\mathbb{Z}/24\mathbb{Z}$, and hence
\begin{equation*}
J_K/24\widehat{\mathbb{Z}}=e\widehat{\mathbb{Z}}/24\widehat{\mathbb{Z}}
\quad\textrm{for some $e\in\mathbb{N}$ dividing $24$}. 
\end{equation*}
Taking the inverse 
image of $e\widehat{\mathbb{Z}}/24\widehat{\mathbb{Z}}$
under the canonical homomorphism $\widehat{\mathbb{Z}}\rightarrow\widehat{\mathbb{Z}}
/24\widehat{\mathbb{Z}}$, 
we obtain $J_K=e\widehat{\mathbb{Z}}$. It follows from (i) that
\begin{equation*}
[\mathcal{G}_K,\,\mathcal{G}_K]=\{\eta_u~|~u\in J_K\}\simeq J_K=e\widehat{\mathbb{Z}}\simeq\widehat{\mathbb{Z}}. 
\end{equation*}
\end{enumerate}
\end{proof}

\begin{lemma}\label{Aut}
Let $H$ be a group isomorphic to $\widehat{\mathbb{Z}}$. Then $H$ has a canonical
$\widehat{\mathbb{Z}}$-module structure via the isomorphism
\begin{equation*}
H\stackrel{\sim}{\rightarrow}\varprojlim_{N\in\mathbb{N}} H/NH,\quad
h\mapsto (h+NH)_{N\in\mathbb{N}}.
\end{equation*}
Every group homomorphism between two such
groups is $\widehat{\mathbb{Z}}$-linear.
Consequently, 
\begin{equation*}
\mathrm{Aut}(H)\simeq\widehat{\mathbb{Z}}^\times,
\end{equation*}
where an element $a\in\widehat{\mathbb{Z}}^\times$ corresponds to multiplication by $a$.
\end{lemma}
\begin{proof}
This follows immediately from the inverse-limit description
of $\widehat{\mathbb{Z}}$. 
See also \cite[Theorem 4.4.7 and Corollary 4.4.8]{R-Z}. 
\end{proof}

\begin{remark}
When the groups in Lemma \ref{Aut} are endowed with their canonical profinite
topologies, every abstract group homomorphism between them is automatically
continuous; see \cite[Theorem 1.1]{N-S}. 
\end{remark}

Since $[\mathcal{G}_K,\,\mathcal{G}_K]\trianglelefteq \mathcal{G}_K$, we get a homomorphism
\begin{equation*}
\theta_K:\mathcal{G}_K\rightarrow\mathrm{Aut}([\mathcal{G}_K,\,\mathcal{G}_K]),
\quad g\mapsto\theta_K(g)
\end{equation*}
where $\theta_K(g)$ is defined by 
\begin{equation*}
\theta_K(g)(x)=gxg^{-1}\quad(x\in[\mathcal{G}_K,\,\mathcal{G}_K]).
\end{equation*}
Furthermore, by Lemmas \ref{comm} (ii) and \ref{Aut} we deduce the canonical isomorphism
\begin{equation*}
\phi_K:\mathrm{Aut}([\mathcal{G}_K,\,\mathcal{G}_K])\stackrel{\sim}{\rightarrow}\widehat{\mathbb{Z}}^\times. 
\end{equation*}
Define a homomorphism
\begin{equation*}
\Theta_K:\mathcal{G}_K\rightarrow\widehat{\mathbb{Z}}^\times
~\textrm{as the composition}~\Theta_K=\phi_K\circ\theta_K. 
\end{equation*}

\begin{lemma}\label{iTIK}
We have $\mathrm{im}(\Theta_K)=I_K=\mathrm{im}(\chi_K)$.
\end{lemma}
\begin{proof}
Recall from Lemma \ref{decomposition} that 
\begin{equation}\label{decom}
\mathcal{G}_K=\mathcal{K}_K\rtimes_{\alpha_K}\mathcal{A}_K.
\end{equation}
Then it follows that
\begin{equation*}
\mathcal{G}_K/\mathcal{K}_K\simeq\mathcal{A}_K\simeq
\mathrm{Gal}(K^\mathrm{ab}/K),
\end{equation*}
which implies the inclusion $[\mathcal{G}_K,\,\mathcal{G}_K]\subseteq\mathcal{K}_K$. 
Moreover, since $\mathcal{K}_K\simeq\widehat{\mathbb{Z}}$ is abelian, 
every element of $\mathcal{K}_K$ acts trivially on $[\mathcal{G}_K,\,\mathcal{G}_K]$ by conjugation, and hence
\begin{equation}\label{trivial}
\mathrm{im}(\Theta_K|_{\mathcal{K}_K})=\{1\}.
\end{equation}
\par
Now, let $\sigma\in\mathcal{A}_K$. By Lemmas \ref{decomposition} and \ref{comm} (i), we have
\begin{equation}\label{tse}
\theta_K(\sigma)(\eta_u)
 =\sigma\eta_u\sigma^{-1}
 =\eta_{\chi_K(\sigma^{-1}|_{K^\mathrm{ab}})u}\quad(u\in J_K).
\end{equation}
By \eqref{ZKue}, Lemmas \ref{comm} and \ref{Aut}, the isomorphism
\begin{equation*}
J_K\stackrel{\sim}{\rightarrow}[\mathcal{G}_K,\,\mathcal{G}_K],\quad u\mapsto\eta_u
\end{equation*}
is $\widehat{\mathbb Z}$-linear. 
And, it follows from \eqref{tse} that
\begin{equation*}
\Theta_K(\sigma)=\chi_K(\sigma^{-1}|_{K^\mathrm{ab}}),
\end{equation*}
which renders by \eqref{AKG} 
\begin{equation}\label{imageA}
\mathrm{im}(\Theta_K|_{\mathcal{A}_K})=I_K.
\end{equation}
Combining \eqref{decom}, \eqref{trivial} and \eqref{imageA}, we conclude
\begin{equation*}
\mathrm{im}(\Theta_K)=I_K.
\end{equation*}
\end{proof}

\begin{theorem}\label{rigidity}
Let $D_1$ and $D_2$ be negative discriminants. Then we derive that
\begin{equation*}
\widehat{C}(D_1)\simeq\widehat{C}(D_2)
~\text{as groups}
\quad\Longleftrightarrow\quad
\mathbb{Q}(\sqrt{D_1})=\mathbb{Q}(\sqrt{D_2}).
\end{equation*}
\end{theorem}
\begin{proof}
Set
\begin{equation*}
K_i=\mathbb{Q}(\sqrt{D_i})\quad\textrm{and}\quad
G_i=\mathrm{Gal}\left(K_i^\mathrm{ab}(\mathfrak{t}^{1/\infty})/K_i(\mathfrak{t})\right)
\quad(i=1,\,2).
\end{equation*}
If $K_1=K_2$, then by Theorem \ref{main} we are certain that
\begin{equation*}
\widehat{C}(D_1)\simeq G_1=G_2\simeq\widehat{C}(D_2).
\end{equation*}
\par
Conversely, suppose that
$\widehat{C}(D_1)\simeq\widehat{C}(D_2)$. By
Theorem \ref{main}, there is an isomorphism
\begin{equation*}
\varphi:G_1\stackrel{\sim}{\rightarrow}G_2,
\end{equation*}
which restricts to an isomorphism
\begin{equation*}
\varphi':[G_1,\,G_1]\stackrel{\sim}{\rightarrow}[G_2,\,G_2].
\end{equation*}
By Lemmas \ref{comm} (ii) and \ref{Aut}, $[G_i,\,G_i]$
are canonically $\widehat{\mathbb{Z}}$-modules ($i=1,\,2$) and $\varphi'$ is
$\widehat{\mathbb{Z}}$-linear.
We claim that
\begin{equation}\label{TvT}
\Theta_{K_2}\circ\varphi=\Theta_{K_1}.
\end{equation}
Indeed, let $g\in G_1$ and
\begin{equation*}
a=\Theta_{K_1}(g)\quad(\in\widehat{\mathbb{Z}}^\times).
\end{equation*}
Then we see that for every
$x\in[G_1,\,G_1]$
\begin{align*}
\theta_{K_2}(\varphi(g))(\varphi'(x))&=\varphi(g)\varphi'(x)\varphi(g)^{-1}\\
&=\varphi(gxg^{-1})\\
&=\varphi'(gxg^{-1})\quad\textrm{since}~gxg^{-1}\in [G_1,\,G_1]\trianglelefteq G_1\\
&=\varphi'(\theta_{K_1}(g)(x))\\
&=\varphi'(ax)\quad\textrm{by Lemma \ref{Aut}}\\
&=a\varphi'(x)\quad\textrm{because $\varphi'$ is $\widehat{\mathbb{Z}}$-linear}.
\end{align*}
Now that $\varphi'$ is surjective, the above observation yields by Lemma \ref{Aut} that
\begin{equation*}
\Theta_{K_2}(\varphi(g))=a=\Theta_{K_1}(g), 
\end{equation*}
which proves \eqref{TvT}. 
So, we establish that
\begin{align*}
I_{K_2}&=\mathrm{im}(\Theta_{K_2})\quad\textrm{by Lemma \ref{iTIK}}\\
&=\mathrm{im}(\Theta_{K_2}\circ\varphi)\quad\textrm{because $\varphi:G_1\rightarrow G_2$
is surjective}\\
&=\mathrm{im}(\Theta_{K_1})\quad\textrm{by \eqref{TvT}}\\
&=I_{K_1}.
\end{align*}
Hence it follows from \eqref{IK} and the injectivity of $\iota$ that
\begin{equation*}
\mathrm{Gal}(\mathbb{Q}^\mathrm{ab}/K_1)=
\mathrm{Gal}(\mathbb{Q}^\mathrm{ab}/K_2).
\end{equation*}
Taking fixed fields in $\mathbb{Q}^\mathrm{ab}$, we obtain $K_1=K_2$.
\end{proof}

\begin{remark}
If $D_1$ and $D_2$ are negative fundamental discriminants, then
Theorem~\ref{rigidity} gives
\begin{equation*}
\widehat{C}(D_1)\simeq\widehat{C}(D_2)~\textrm{as groups}
\quad\Longleftrightarrow\quad D_1=D_2.
\end{equation*}
\end{remark}

\begin{remark}
Let $K_1$ and $K_2$ be arbitrary number fields,
and let $\overline{K}_1$ and
$\overline{K}_2$ be their respective algebraic closures. One
consequence of the Neukirch-Uchida theorem \cite{Neukirch69, Neukirch692, Uchida}
is that
\begin{equation*}
\mathrm{Gal}(\overline{K}_1/K_1)\simeq
\mathrm{Gal}(\overline{K}_2/K_2)~\textrm{as profinite groups}
\quad\Longleftrightarrow\quad K_1\simeq K_2. 
\end{equation*}
This result was substantially strengthened in the direction of
solvable quotients by Sa\"{i}di and Tamagawa \cite{S-T}, who proved
that if the maximal three-step solvable Galois groups of 
$K_1$ and $K_2$ are isomorphic as profinite groups, then $K_1$ and $K_2$
are themselves isomorphic.
For the broader context of anabelian geometry,
we refer the reader to \cite{Grothendieck, Mochizuki, Stix}; see also \cite{B-T} for an
introduction to the birational anabelian viewpoint.
\end{remark}

\section*{Statements \& Declarations}

\subsection*{Funding}
The first author was supported
by the National Research Foundation of Korea (NRF) grant funded by the Korea government (MSIT)
(No. RS-2024-00348504).
The second (corresponding) author was supported
by Hankuk University of Foreign Studies Research Fund of 2026.
The third author was supported
by the National Research Foundation of Korea (NRF) grant funded by the Korea government (MSIT)
(No. RS-2024-00342522).

\subsection*{Competing interests}
The authors have not disclosed any competing interests.

\subsection*{Data availability}
Data sharing is not applicable to this article as no datasets
were generated or analysed during the current study.

\bibliographystyle{amsplain}

\address{
Department of Mathematical Sciences \\
KAIST \\
Daejeon 34141\\
Republic of Korea} {jkgoo@kaist.ac.kr}

\address{
Department of Mathematics\\
Hankuk University of Foreign Studies\\
Yongin-si, Gyeonggi-do 17035\\
Republic of Korea} {dhshin@hufs.ac.kr}

\address{
Department of Mathematics Education\\
Pusan National University\\
Busan 46241\\Republic of Korea}
{dsyoon@pusan.ac.kr}

\end{document}